\documentclass{amsart}

\title{The formal shift operator on the Yangian double}
\author[C. Wendlandt]{Curtis Wendlandt}
\address{Department of Mathematics, The Ohio State University.}
\email{wendlandt.4@osu.edu}

\makeatletter
\@namedef{subjclassname@2020}{%
  \textup{2020} Mathematics Subject Classification}
\makeatother


\usepackage[lite]{amsrefs} 
\usepackage{amsmath,amssymb} 
\usepackage{mathtools}
\usepackage[usenames,dvipsnames]{xcolor} 
\usepackage{bbm} 
\usepackage[scr=boondoxo]{mathalfa} 
\usepackage{dsfont} 
\usepackage{enumitem}
\usepackage{version} 
\usepackage[all]{xy} 
\usepackage{tikz-cd} 
\usepackage{accents} 
\usepackage{stmaryrd} 

\usepackage[
	colorlinks,
	plainpages,
	citecolor=red!50!black,
    filecolor=Darkgreen,
    linkcolor=blue!40!black,
    urlcolor=cyan!50!black!90]{hyperref} 
    
\newtheorem{theorem}{Theorem}[section]
\newtheorem{proposition}[theorem]{Proposition}
\newtheorem{corollary}[theorem]{Corollary}
\newtheorem{lemma}[theorem]{Lemma}
\theoremstyle{definition}
\newtheorem{definition}[theorem]{Definition}
\newtheorem{remark}[theorem]{Remark}

\newcommand{\wh}{\widehat}

\newcommand{\End}{\mathrm{End}}

\newcommand{\Ker}{\mathrm{Ker}}
\newcommand{\iso}{\xrightarrow{\,\smash{\raisebox{-0.5ex}{\ensuremath{\scriptstyle\sim}}}\,}}

\newcommand{\into}{\hookrightarrow}
\newcommand{\onto}{\twoheadrightarrow}

\newcommand{\mfa}{\mathfrak{a}}

\newcommand{\mfb}{\mathfrak{b}}

\newcommand{\mfg}{\mathfrak{g}}
\newcommand{\mfh}{\mathfrak{h}}

\newcommand{\mfp}{\mathfrak{p}}

\newcommand{\mfs}{\mathfrak{s}}

\newcommand{\mft}{\mathfrak{t}}

\newcommand{\mfz}{\mathfrak{z}}

\newcommand{\mfgl}{\mathfrak{g}\mathfrak{l}}
\newcommand{\mfsl}{\mathfrak{s}\mathfrak{l}}

\newcommand{\mcA}{\mathcal{A}}

\newcommand{\mcH}{\mathcal{H}}

\newcommand{\mcJ}{\mathcal{J}}

\newcommand{\mcR}{\mathcal{R}}

\newcommand{\mcX}{\mathcal{X}}
\newcommand{\mcY}{\mathcal{Y}}

\newcommand{\mbA}{\mathbf{A}}

\newcommand{\mbI}{\mathbf{I}}

\newcommand{\mbt}{\mathbf{t}}

\newcommand{\C}{\mathbb{C}}

\newcommand{\Z}{\mathbb{Z}}

\newcommand{\msA}{\mathsf{A}}

\newcommand{\msB}{\mathsf{B}}

\newcommand{\msC}{\mathsf{C}}

\newcommand{\msD}{\mathsf{D}}

\newcommand{\veps}{\varepsilon}

\allowdisplaybreaks

\numberwithin{equation}{section}

\newcommand{\N}{\mathbb{N}}
\newcommand{\Yhg}{Y_\hbar\mfg}
\newcommand{\Yhdg}{Y_\hbar\dot\mfg}
\newcommand{\Yhgp}{\Yhg_{\scriptscriptstyle +}}
\newcommand{\Yhb}[1]{Y_\hbar(\mfb_{\scriptscriptstyle {#1}})}
\newcommand{\cYhg}{\widehat{Y_\hbar\mfg}}
\newcommand{\cYhgp}{\cYhg_{\scriptscriptstyle +}}
\newcommand{\eYhg}{\cYhg_{\scriptscriptstyle \times}}
\newcommand{\cYhgeq}[1]{\widehat{Y_\hbar\mfg}_{{\scriptscriptstyle\geq}{#1}}}
\newcommand{\cYg}[1]{\widehat{Y_{#1}\mfg}}
\newcommand{\Yhgz}{\cYhg_z}
\newcommand{\LzYhg}{\mathds{L}\cYhg_z}
\newcommand{\LzYgr}[1]{\mathds{L}\cYhg_{z,{#1}}}
\newcommand{\LzhYhg}{\mathrm{L}\cYhg_z}

\newcommand{\DYhg}{\mathrm{D}Y_\hbar\mfg}
\newcommand{\DYhdg}{\mathrm{D}Y_\hbar\dot\mfg}
\newcommand{\DYg}{\mathds{D}Y_{\hbar}\mfg}
\newcommand{\DYdg}{\mathds{D}Y_{\hbar}\dot\mfg}
\newcommand{\cDYhg}{\widehat{\mathrm{D}Y_\hbar\mfg}}

\newcommand{\mcJY}{\mathcal{J}_{\scriptscriptstyle{+}}}

\newcommand{\cUsh}{\widehat{U(\mfs_\mfh)}}
\newcommand{\cUs}{\widehat{U(\mfs)}}
\newcommand{\gfin}{\bar{\mfg}}
\newcommand{\uce}{\mathfrak{uce}}

\newcommand{\jD}{\jmath} 
\newcommand{\iY}{\imath_{\mathrm{Y}}} 
\newcommand{\iYh}{\imath} 
\newcommand{\evg}{\mathrm{ev}_{\!\mfg}}
\newcommand{\evdg}{\dot{\mathrm{ev}}_{\!\mfg}}  
\newcommand{\id}{\mathbf{1}}

\begin{document}

\begin{abstract}
Let $\mfg$ be a symmetrizable Kac--Moody algebra with associated Yangian $\Yhg$ and Yangian double $\DYhg$.  An elementary result of fundamental importance to the theory of Yangians is that, for each $c\in \C$, there is an automorphism $\tau_c$ of $\Yhg$ corresponding to the translation $t\mapsto t+c$ of the complex plane. 
Replacing $c$ by a formal parameter $z$ yields the so-called formal shift homomorphism $\tau_z$ from $\Yhg$ to the polynomial algebra $\Yhg[z]$. 

We prove that $\tau_z$ uniquely extends to an algebra homomorphism
$\Phi_z$  from the Yangian double $\DYhg$ into the $\hbar$-adic closure of the  algebra of Laurent series in $z^{-1}$ with coefficients in the Yangian $\Yhg$. This induces, via evaluation at any point $c\in \C^\times$, a homomorphism from $\DYhg$ into the completion of the Yangian with respect to its grading. We show that each such homomorphism gives rise to an isomorphism between completions of $\DYhg$ and $\Yhg$ and,  as a corollary, we find that the Yangian $\Yhg$ can be realized as a degeneration of the Yangian double $\DYhg$.  Using these results, we obtain a Poincar\'{e}--Birkhoff--Witt theorem for $\DYhg$ applicable when $\mfg$ is of finite type or of simply-laced affine type. 
\end{abstract}

\maketitle

{\setlength{\parskip}{0pt}
\setcounter{tocdepth}{1} 
\tableofcontents
}

\section{Introduction}\label{sec:Intro}

\subsection{}

In this article, we study the Yangian double $\DYhg$ associated to a symmetrizable Kac--Moody algebra $\mfg$, after Khoroshkin and Tolstoy \cite{KT96}, by taking the approach that it should be characterized in terms of the underlying Yangian $\Yhg$. Our main results 
realize such a characterization by showing that $\DYhg$ can be viewed as both a dense subalgebra of the completion of the Yangian $\Yhg$ with respect to its $\N$-grading, and as the closure of a $\Z$-graded subalgebra of the space of formal Laurent series in $z^{-1}$ with coefficients in $\Yhg$. As a particular consequence of this description, we obtain a uniform Poincar\'{e}--Birkhoff--Witt theorem for the Yangian double $\DYhg$ of an arbitrary finite-dimensional or simply laced affine Kac--Moody algebra. 
These results are based on the construction of an extension $\Phi_z$ of the formal shift homomorphism $\tau_z$ on the Yangian to the Yangian double $\DYhg$. Here we recall that $\tau_z$ is a  graded algebra embedding 
\begin{equation*}
\tau_z:\Yhg\to \Yhg[z]
\end{equation*}
which gives rise to an action of the group of translations of the complex plane on $\Yhg$; see \eqref{shift-c} and \eqref{shift-z}. This action dates back to the foundational work of Drinfeld \cite{Dr} and
has become ubiquitous in the theory of Yangians.

\subsection{}
For the purpose of motivating our construction, let us first consider its classical counterpart with $\mfg$ taken to be a complex semisimple Lie algebra. Under this assumption, the Yangian $\Yhg$ and Yangian double $\DYhg$ are graded deformations of the enveloping algebras for the current algebra $\mfg[t]$ and loop algebra $\mfg[t,t^{-1}]$, respectively, and  $\tau_z$ provides a quantization of the embedding  
\begin{equation*}
\gamma_z:\mfg[t]\to \mfg[t,z]
\end{equation*}
sending any polynomial $f(t)$ to its translate $f(t+z)$. Note that this is a graded homomorphism, provided $t$ and $z$ are both given degree $1$. As $t+z$ is an invertible element in the ring $\C[t][z;z^{-1}]\!]$ of Laurent series in $z^{-1}$ with coefficients in $\C[t]$, $\gamma_z$ uniquely extends to a graded Lie algebra homomorphism  
\begin{equation*}
\Upsilon_z:\mfg[t,t^{-1}]\to \bigoplus_{n\in \Z}z^n \mfg[\![t/z]\!]\subset \mfg[t][z;z^{-1}]\!],
\end{equation*}
where $\mfg[\![t/z]\!]$ is the image of the embedding $\mfg[\![w]\!]\to \mfg[t][z;z^{-1}]\!]$ sending $f(w)$ to $f(t/z)$. This homomorphism is injective and possesses a number of properties which elucidate the intimate connection shared by $\mfg[t]$ and $\mfg[t,t^{-1}]$. For instance, the formal parameter $z$ may be evaluated to any $c\in \C^\times$ to yield a family of Lie algebra embeddings 
\begin{equation*}
\Upsilon_c: \mfg[t,t^{-1}]\to \mfg[\![t]\!].
\end{equation*}
Each member $\Upsilon_c$ of this family restricts to an automorphism of $\mfg[t]$ and uniquely extends to an isomorphism
\begin{equation*}
\wh{\Upsilon}_{c}: \wh{\mfg[t,t^{-1}]}\iso \mfg[\![t]\!]
\end{equation*}
when $\mfg[t,t^{-1}]$ is completed with respect to the descending filtration given by the lower central series for the evaluation ideal $\mathds{J}_c=(t-c)\mfg[t,t^{-1}]$. In addition, $\Upsilon_c$ induces an isomorphism of $\N$-graded Lie algebras 
\begin{equation*}
\mathrm{gr}(\Upsilon_c):\mathrm{gr}(\mfg[t,t^{-1}])\iso \bigoplus_{n\geq 0} t^n\mfg[\![t]\!]/t^{n+1}\mfg[\![t]\!] \cong \mfg[t]
\end{equation*}
which realizes $\mfg[t]$ as a degeneration of $\mfg[t,t^{-1}]$.

The results of the present paper provide a quantization $\Phi_z$ of $\Upsilon_z$ admitting counterparts to each of the above properties and satisfying the commutative diagram
\tikzcdset{every label/.append style = {font = \small}}
\begin{equation}\label{dia:Intro}
\begin{tikzcd}[column sep=6ex, row sep=7ex]
 \DYhg  \arrow{rr}{\Phi_z} && \LzhYhg\\
 & \Yhg \arrow{lu}[below left]{\iYh} \arrow{ru}[below right]{\tau_z} &       
\end{tikzcd}
\end{equation}
where $\iYh$ is a quantization of  the natural inclusion $\mfg[t]\subset \mfg[t,t^{-1}]$ and $\LzhYhg$ is the $\hbar$-adic completion of the $\Z$-graded subalgebra 
\begin{equation*}
\bigoplus_{n\in \Z}z^n \Yhgz \subset \Yhg[z;z^{-1}]\!],
\end{equation*}
where $\Yhgz=\prod_{k\in \N}\Yhg_kz^{-k}$ and $\Yhg_k$ is the $k$-th graded component of $\Yhg$. This completed algebra is described explicitly in Proposition \ref{P:whYgz} and plays the role of the graded Lie algebra $\bigoplus_{n\in \Z}z^n \mfg[\![t/z]\!]$ in the above classical picture. 

\subsection{} 
When $\mfg$ is an infinite-dimensional symmetrizable Kac--Moody algebra, the Yangian $\Yhg$ and Yangian double $\DYhg$ no longer deform the enveloping algebras of the respective Lie algebras $\mfg[t]$ and $\mfg[t,t^{-1}]$. They do, however, deform the enveloping algebras of semidirect products 
\begin{equation*}
\mfs\rtimes \ddot\mfh \quad \text{ and } \quad  \mft\rtimes \ddot\mfh,
\end{equation*}
where $\ddot\mfh$ is a finite-dimensional abelian subalgebra of $\mfg$, and $\mfs$ and $\mft$ are perfect Lie algebras which project onto the current algebra $\dot\mfg[t]$ and loop algebra $\dot\mfg[t,t^{-1}]$, respectively, of the derived subalgebra $\dot\mfg=[\mfg,\mfg]$. In general, the kernel of these projections is large and one cannot a priori put too much stock in the classical story outlined above.
Despite this fact, our construction of $\Phi_z$ remains completely valid, and we exploit it as an effective and simple algebraic tool for studying the Yangian double $\DYhg$ in full generality. In fact, one of the main goals of our work is to further develop the algebraic theory of $\DYhg$ when $\mfg$ is an untwisted affine Lie algebra. In this case, $\mfs$ and $\mft$ are non-trivial central extensions of $\dot\mfg[t]$ and $\dot\mfg[t,t^{-1}]$, respectively, which admit entirely concrete descriptions, and $\DYhg$ may be viewed as a  rational, level zero, analogue of the so-called quantum toroidal algebra associated to $\mfg$.

The associated affine Yangians were first studied in detail in the work of Guay \cites{Gu05,Gu07,Gu09} in type $\msA$, which in particular illuminated their connection to both  rational and trigonometric Cherednik algebras, as well as deformed double current algebras. They have since been afforded a more general treatment in what is now a rapidly growing body of literature. This includes, but is not limited to, the list of contributions \cites{YaGuPBW,YaGu3,YaGu1, GNW,GRWvrep,Kod18,Kod19,Kod19b,BerTsy19,Tsy17,Tsy17b}.

\subsection{}\label{I:Master} Let us now outline our main results in detail. Let $\mfg$ be a symmetrizable Kac--Moody algebra, and let $\cYhg$ denote the formal completion of the Yangian $\Yhg$ with respect to its $\N$-grading. In this article, we prove the following theorem. 
\begin{theorem}\label{T:Master1}
There is a unique algebra homomorphism  $\Phi_z:\DYhg\to \LzhYhg$
satisfying the commutative diagram \eqref{dia:Intro}. Moreover:
\begin{enumerate}[font=\upshape]
\item\label{Master:1} $\Phi_z$ is a $\Z$-graded algebra homomorphism.

\item\label{Master:2} $\Phi_z$ evaluates at any $z=c\in \C^{\times}$ to an algebra homomorphism 
\begin{equation*}
\Phi_c:\DYhg\to \cYhg 
\end{equation*}
which uniquely extends the shift automorphism $\tau_c$ of $\Yhg$.
\item\label{Master:3} Each specialization $\Phi_c$ of $\Phi_z$ determines an isomorphism
\begin{equation*}
\wh{\Phi}_c:\cDYhg_c\iso \cYhg,
\end{equation*}

where $\DYhg$ is completed with respect to its evaluation ideal at $t=c$. 

\item\label{Master:4} $\Phi_z$ induces an isomorphism of $\N$-graded algebras
\begin{equation*}
\mathrm{gr}(\DYhg)\iso \Yhg,
\end{equation*}
where $\DYhg$ is filtered by powers of its evaluation ideal at $t=1$. 

\end{enumerate}
\end{theorem}
This theorem is the amalgamation of two of the three main results established in this paper. Our first main result, Theorem \ref{T:Phi},  outputs our main tool: a unique extension $\Phi_z$ of $\tau_z$ satisfying \eqref{Master:1} and \eqref{Master:2}. 
Our second main result, Theorem \ref{T:hatiso}, then proves that $\Phi=\Phi_1$ induces an isomorphism 
\begin{equation*}
\wh{\Phi}:\cDYhg\iso \cYhg, 
\end{equation*} 
where $\cDYhg$ is the completion of $\DYhg$ with respect to its evaluation ideal $\mcJ$ at $t=1$.
This is precisely the assertion of \eqref{Master:3} in the special case where $c=1$, and is generalized to an arbitrary evaluation point $c\in \C^\times$ in Corollary \ref{C:hatiso}, using that $\Phi$ may be transformed into $\Phi_c$ by conjugating by a gradation automorphism, as proven in Proposition \ref{P:gr-tw}.

Part \eqref{Master:4} of the above theorem provides the Yangian double analogue of Drinfeld's result \cite{DrQG}, proven by Guay and Ma in \cite{GM}, that the Yangian $\Yhg$ may be realized as a degeneration of the quantum loop algebra $U_\hbar(L\mfg)$. It is established in Corollary \ref{C:degen} as an application of Theorem \ref{T:hatiso}. 
As another byproduct of Theorem \ref{T:hatiso}, we find in Corollary \ref{C:YJ->DJ} that $\iYh$ extends to an isomorphism 
\begin{equation*}
\wh{\iYh}:\cYhg_{\scriptscriptstyle \times}\iso \cDYhg, 
\end{equation*}
where $\eYhg$ is the completion of $\Yhg$ with respect to its own evaluation ideal at $t=1$. As explained in Section \ref{ssec:ev-Yhg}, this affords the completed Yangian double a rather precise description.

Our third main result is provided by Theorem \ref{T:PBW}, which outputs the following Poincar\'{e}--Birkhoff--Witt theorem for $\DYhg$. 
\begin{theorem}\label{T:Master2}
Let $\mfg$ be a symmetrizable Kac--Moody algebra of finite type or of simply-laced affine type. Then:
\begin{enumerate}[font=\upshape]
\item $\Phi_z$ and $\Phi_c$ are injective for each $c\in \C^\times$.
\item\label{MasterPBW:2} $\DYhg$ is a flat deformation of $U(\mft\rtimes \ddot\mfh)$ over $\C[\![\hbar]\!]$. In particular, there is an isomorphism of $\C[\![\hbar]\!]$-modules 
$
\DYhg\cong U(\mft\rtimes\ddot\mfh)[\![\hbar]\!]
$.
\end{enumerate}
\end{theorem}
When $\mfg$ is finite-dimensional, the abelian Lie algebra $\ddot\mfh$ vanishes and $\mft$ coincides with the loop algebra $\mfg[t,t^{-1}]$. In this case, Part \eqref{MasterPBW:2} of this theorem improves upon \cite{Enr03}*{Thm.~1.5}, which established
that the positive part $\mathrm{D}Y_\hbar^{\scriptscriptstyle +}\!\mfg$ of the Yangian double $\DYhg$ is topologically free. It is also a close relative of \cite{EnQHA}*{Prop.~5.4} which, in the particular setting outlined in \cite{EnQHA}*{Rem.~8}, shows that a quantum algebra closely related to the so-called centrally extended Yangian double \cite{Khor95} has a similar flatness property. When $\mfg$ is taken to be a classical Lie algebra of type $\msB$, $\msC$ or $\msD$, Part \eqref{MasterPBW:2} of Theorem \ref{T:Master2} is in fact  a consequence of Theorems 3.4 and 6.2 from the recent article \cite{JiYaLi20}. 
In the type $\msA$ setting, this should instead follow from the Poincar\'{e}--Birkhoff--Witt result established in Theorem 2.2 of \cite{JKMY18} (see also \cite{Naz20}*{Thm.~15.3}) for the Yangian double of $\mfgl_N$ in its $R$-matrix presentation, and the identification obtained in \cite{Io96}*{Cor.~3.5}. The proof given in the present paper does not rely on these results, and applies uniformly in all Dynkin types.

In the affine setting, there does not appear to be any counterpart to either part of Theorem \ref{T:Master2} which exists in the literature. 
Our proof applies the recent results of \cites{GRWvrep} and \cite{YaGuPBW}, and ultimately reduces to a 
detailed computation of the classical limit of $\Phi$, which we prove is injective under the more general hypothesis that $\mfg$ is of untwisted affine type with underlying simple Lie algebra $\bar\mfg\ncong\mfsl_2$. Our arguments exploit the fact that, for any such $\mfg$,  the Lie algebras $\mfs$ and $\mft$ admit perfectly tangible descriptions. Namely, due to a result of Moody, Rao and Yokonuma \cite{MRY90}, one has
\begin{equation*}
\mfs\cong \mathfrak{uce}(\dot\mfg[t])\cong \mathfrak{uce}(\bar\mfg[v^{\pm 1},t]) \quad \text{ and }\quad \mft_\kappa \cong \mathfrak{uce}(\dot\mfg[t,t^{-1}])\cong \mathfrak{uce}(\bar\mfg[v^{\pm 1},t^{\pm 1}]),
\end{equation*} 
where $\mft_\kappa$ is a one-dimensional central extension of $\mft$, and $\mathfrak{uce}(\mfa)$ denotes the universal central extension of a given perfect Lie algebra $\mfa$. Universal central extensions of this type were realized concretely in the work of Kassel \cite{Kas84}, and this description is recalled in the course of our proof of Theorem \ref{T:PBW}: see Sections \ref{ssec:aff-I}--\ref{ssec:aff-II}.

\subsection{}

The results obtained in this paper, coupled with the findings of \cite{GTLW19}, lay the foundation for a uniform proof of a conjecture from the pioneering work \cite{KT96} of Khoroshkin and Tolstoy. This is the assertion that, when $\mfg$ is finite-dimensional, $\DYhg$ coincides with the restricted quantum double of the Yangian $\Yhg$. 

Our interest in this conjecture stems, in part, from a desire to understand the universal $R$-matrix  of the Yangian from a more familiar Hopf-theoretic point of view. This is a remarkable formal series $\mcR(z)\in (\Yhg\otimes \Yhg)[\![z^{-1}]\!]$, introduced by Drinfeld in \cite{Dr}, which has played a central role in many of the developments at the heart of the representation theory of Yangians. It is not, however, understood to be a universal $R$-matrix in the traditional sense and, in particular, has not been shown to arise as the canonical tensor associated to a Hopf pairing. On the other hand, the universal $R$-matrix $\mathds{R}$ associated to the restricted quantum double of the Yangian $\Yhg$ has these properties by construction.

In the sequel \cite{WRQD} to this paper, we will show that there is a unique Hopf algebra structure on $\DYhg$ preserved by $\Phi_z$ and that, when equipped with this structure, $\DYhg$ is isomorphic to the restricted quantum double of $\Yhg$, as conjectured in \cite{KT96}. Using this identification, we will 
establish that $\mathds{R}$ and $\mcR(z)$ are in fact one and the same.
More precisely, one has the equality
\begin{equation*}
(\Phi_v\otimes \Phi_z) \mathds{R}=\mcR(v-z)\in (\Yhg\otimes\Yhg)[v][\![z^{-1}]\!].
\end{equation*}
%
%

\subsection{} 
Let us now consider the situation in which $\hbar$ is replaced with a nonzero complex number $\mu\in \C^\times$. Let $\mfg$ be a finite-dimensional simple Lie algebra with associated Yangian $Y_\mu\mfg=\Yhg/(\hbar-\mu)\Yhg$. The Yangian double $\DYhg$ itself admits a $\C[\hbar]$-form $\DYg$ (see Definition \ref{D:DY}) which may be specialized to obtain a $\C$-algebra $\mathrm{D}Y_\mu(\mfg)=\DYg/(\hbar-\mu)\DYg$. Let 
$\mathrm{Rep}_{fd}(Y_\mu\mfg)$ and $\mathrm{Rep}_{fd}(\mathrm{D}Y_\mu\mfg)$ denote the categories of finite-dimensional representations of $Y_\mu\mfg$ and $\mathrm{D}Y_\mu\mfg$, respectively.

The results of this article have recently been applied in \cite{GWPoles}*{\S5} to construct, for each $c\in \C$, an equivalence of categories 
\begin{equation*}
\Theta_c:\mathrm{Rep}_{fd}^c(Y_\mu\mfg)\iso \mathrm{Rep}_{fd}(\mathrm{D}Y_\mu\mfg),
\end{equation*}
where $\mathrm{Rep}_{fd}^c(Y_\mu\mfg)$ is the full subcategory of $\mathrm{Rep}_{fd}(Y_\mu\mfg)$ consisting of all $V$ with the property that the commuting Cartan currents $\{h_i(u)\}_{i\in \mbI}\subset Y_\mu\mfg[\![u^{-1}]\!]$, defined in Proposition \ref{P:Y-op} below, have poles contained in the punctured complex plane $\C\setminus\{-c\}$ when viewed as $\End V$-valued rational functions\footnote{By \cite{GTL2}*{Prop.~3.6}, each $h_i(u)$ necessarily operates on $V$ as the Taylor expansion at $u=\infty$ of an operator valued rational function of $u$.} of $u$. 
%

When $c\in \C^\times$, the functor $\Theta_c$ can be interpreted as the restriction of the pull-back functor $\Phi_c^\ast$ to $\mathrm{Rep}_{fd}^c(Y_\mu\mfg)$ upon specializing $\hbar$ to $\mu$. In more detail, the homomorphism $\Phi_z$ from Theorem \ref{T:Master1} admits a specialization 
\begin{equation*}
\Phi_z^\mu:\mathrm{D}Y_\mu\mfg\to Y_\mu\mfg[z;z^{-1}]\!]
\end{equation*}
and the $\mathrm{D}Y_\mu(\mfg)$-module $\Theta_c(V)$ is obtained from ${(\Phi_z^\mu)}^\ast(V[z;z^{-1}]\!])$ by evaluating $z$ at the point $c$. That such an evaluation is permitted is a consequence of the definition of $\mathrm{Rep}_{fd}^c(Y_\mu\mfg)$; we refer the reader to \cite{GWPoles}*{\S5} for complete details. 

It is not difficult to generalize the construction of $\Theta_c$ from \cite{GWPoles} to the setting where $\mfg$ is an arbitrary symmetrizable Kac--Moody algebra. In this generality, $\mathrm{Rep}_{fd}(Y_\mu\mfg)$ and $\mathrm{Rep}_{fd}(\mathrm{D}Y_\mu\mfg)$ are replaced with the categories of $Y_\mu\mfg$ and $\mathrm{D}Y_\mu\mfg$ modules whose restrictions to $\mfg$ are integrable and in the category $\mathcal{O}$. In fact, one may even take the larger categories consisting of all $\mfh$-diagonalizable $Y_\mu\mfg$ and $\mathrm{D}Y_\mu\mfg$ modules with finite-dimensional weight spaces, where $\mfh\subset \mfg$ is a fixed Cartan subalgebra.

\subsection{} 

To conclude, it should be emphasized that the approach taken in this article both complements and draws inspiration from the innovative work \cite{GTL1} of Gautam and Toledano Laredo. Therein, the authors constructed a highly non-trivial algebra homomorphism 
\begin{equation*}
\Phi_{\scriptscriptstyle{\mathsf{GTL}}}:U_\hbar(L\mfg)\to \cYhg
\end{equation*}
which has several remarkable properties. In particular, when $\mfg$ is finite-dimensional, it induces isomorphisms
\begin{equation*}
\wh{\Phi}_{\scriptscriptstyle{\mathsf{GTL}}}:\wh{U_\hbar(L\mfg)}\iso \cYhg \quad \text{ and }\quad \mathrm{gr}(\Phi_{\scriptscriptstyle{\mathsf{GTL}}}):\mathrm{gr}(U_\hbar(L\mfg))\iso \Yhg,
\end{equation*}
where $U_\hbar(L\mfg)$ is both completed and filtered with respect to its evaluation ideal at $t=1$: see Theorem 6.2 and Proposition 6.5 of \cite{GTL1}. Combining Theorem \ref{T:Master1} with the results of \cite{GTL1}, we obtain an algebra homomorphism
\begin{equation*}
\Psi=\wh{\Phi}^{-1}\circ \Phi_{\scriptscriptstyle{\mathsf{GTL}}}:U_\hbar(L\mfg)\to \cDYhg
\end{equation*}
which extends to an isomorphism between the evaluation completions of $U_\hbar(L\mfg)$ and $\DYhg$. It may be viewed as a filtered map with associated graded map providing an isomorphism between $\mathrm{gr}(U_\hbar(L\mfg))$ and $\mathrm{gr}(\DYhg)$, both of which may be identified with the Yangian $\Yhg$.
As $U_\hbar(L\mfg)$ and $\DYhg$ both deform the enveloping algebra of the loop algebra $\mfg[t,t^{-1}]$, it is perhaps natural to speculate on whether or not this composition can be viewed as an isomorphism between $U_\hbar(L\mfg)$ and $\DYhg$, without any completions at play. Though we do not consider $\Psi$ in any detail in the present article, we note in passing that this is easily seen not to be the case, even after reducing modulo $\hbar$.

\subsection{Outline}
In Section \ref{sec:Def}, we review the definitions and basic properties of the Yangian $\Yhg$ and Yangian double $\DYhg$ associated to a symmetrizable Kac--Moody algebra $\mfg$. 
 Our preliminary overview continues in Section \ref{sec:Der-CL}, where we introduce the Yangian $\Yhdg$ and Yangian double $\DYhdg$ of $\dot\mfg=[\mfg,\mfg]$, in addition to the Lie algebras  $\mfs\rtimes \ddot\mfh$ and $\mft\rtimes \ddot\mfh$. 
In Section \ref{sec:Phi}, we construct the unique extension $\Phi_z$ of $\tau_z$ and its specialization $\Phi_c$ at any invertible complex number $c$. 
We then show in Section \ref{sec:Iso} that each homomorphism $\Phi_c$ induces an isomorphism between the evaluation completion of $\DYhg$ at the point $c$ and the completion of $\Yhg$ with respect to its natural $\N$-grading. In Section \ref{sec:PBW}, we prove our final main result, which simultaneously establishes the injectivity of $\Phi_z$ and $\Phi_c$, for any $c\in \C^\times$, and the Poincar\'{e}--Birkhoff--Witt theorem for $\DYhg$, when $\mfg$ is of finite type or simply-laced affine type. Finally, Appendix \ref{App:A} contains the proof of a technical result on grading completions used in the proof of Lemma  \ref{L:whYhg} of Section \ref{ssec:CY}.


\subsection{Acknowledgments}
The author gratefully acknowledges the support of the Natural Sciences and Engineering Research Council of Canada (NSERC) provided via the postdoctoral fellowship (PDF) program. He would also like to thank Sachin Gautam for several helpful comments and insightful discussions.

\section{Yangians and Yangian doubles}\label{sec:Def}

Let $\mfg$ be a symmetrizable Kac--Moody algebra with indecomposable Cartan matrix $\mbA=(a_{ij})_{i\in \mbI}$. 
We fix a realization $(\mfh,\{\alpha_i\}_{i\in \mbI},\{\alpha_i^\vee\}_{i\in \mbI})$ of $\mbA$ as in \cite{KacBook90}*{\S1.1}. That is, $\mfh$ is a Cartan subalgebra of $\mfg$, $\{\alpha_i\}_{i\in \mbI}\subset \mfh^*$ is the set of simple roots, and $\{\alpha_i^\vee\}_{i\in \mbI}\subset \mfh$ the set of simple coroots, so that $\alpha_j(\alpha_i^\vee)=a_{ij}$ for all $i,j\in \mbI$. Let $Q=\bigoplus_{i\in \mbI}\Z\alpha_i\subset \mfh^\ast$ be the associated root lattice, and let $(\, ,\,)$ be a standard invariant form on $\mfg$, as in \cite{KacBook90}*{\S2}. We will use the same notation for the induced bilinear form on $\mfh^*$. Set 
\begin{equation*}
 d_{ij}=\frac{(\alpha_i,\alpha_j)}{2} \;\text{ and }\; d_i=d_{ii} \quad \forall \; i,j\in \mbI. 
\end{equation*}
By \cite{KacBook90}*{\S2.3}, we may assume that $(\,,\,)$ is normalized so that $\{d_i\}_{i\in \mbI}$ are positive, relatively prime, integers. 

Let $\dot\mfg$ denote the derived subalgebra $[\mfg,\mfg]$. The notation $\N$ and $\N_+$ will be used to denote the sets of non-negative and strictly positive integers, respectively. All of this data shall remain fixed throughout the course of this paper, unless specified otherwise. 

\subsection{The Yangian \texorpdfstring{$\Yhg$}{Yg}}

We begin by recalling the definition of the Yangian associated to $\mfg$. Let $S_m$ denote the symmetric group on $\{1,\ldots,m\}$. 
\begin{definition}\label{D:Y} The Yangian
 $\Yhg$ is the unital associative $\C[\hbar]$-algebra generated by $h\in \mfh$ and $\{x_{ir}^\pm, h_{ir}\}_{i\in \mbI,r\in \N}$, subject to the following relations for $i,j\in \mbI$, $r,s\in \N$ and $h,h'\in \mfh$:
\begin{gather}
 h_{i0}=d_i\alpha_i^\vee, \label{Y:hi=ai}\\
 [h_{ir},h_{js}]=0, \quad [h_{ir},h]=0,\quad [h,h']=0,\label{Y:hh}\\
 [h,x_{js}^\pm]=\pm \alpha_j(h)x_{js}^\pm, \label{Y:h0x}\\
 [x_{ir}^+,x_{js}^-]=\delta_{ij} h_{i,r+s}, \label{Y:xxh}\\
 [h_{i,r+1},x_{js}^\pm]-[h_{ir},x_{j,s+1}^\pm]=\pm \hbar d_{ij}(h_{ir}x_{js}^\pm+x_{js}^\pm h_{ir}),\label{Y:xh}\\
 [x_{i,r+1}^\pm,x_{js}^\pm]-[x_{ir}^\pm,x_{j,s+1}^\pm]=\pm \hbar d_{ij}(x_{ir}^\pm x_{js}^\pm+x_{js}^\pm x_{ir}^\pm ), \label{Y:xx}\\
 \sum_{\pi \in S_{m}} \left[x_{i,r_{\pi(1)}}^{\pm}, \left[x_{i,r_{\pi(2)}}^{\pm}, \cdots, \left[x_{i,r_{\pi(m)}}^{\pm},x_{js}^{\pm}\right] \cdots\right]\right] = 0, \label{Y:Serre}
\end{gather}
where in the last relation $i\neq j$, $m=1-a_{ij}$ and $r_1,\ldots,r_m\in \N$. 
\end{definition}
The Yangian $\Yhg$ is an $\N$-graded algebra with $\deg\hbar=1$, $\deg\mfh=0$, and 
\begin{equation*}
\deg x_{ir}^\pm=\deg h_{ir}=r \quad \forall \;i\in \mbI,\; r\in \N.
\end{equation*}
The $k$-th graded component of $\Yhg$ will be denoted $\Yhg_k$, so that
\begin{equation*}
\Yhg=\bigoplus_{k\in \N}\Yhg_k. 
\end{equation*}
As a $\C[\hbar]$-algebra, $\Yhg$ is generated by its degree zero and one subspaces. More precisely, we have the following standard result. 
\begin{lemma}\label{L:deg1}
$\Yhg$ is generated by $\mfh\cup\{x_{i0}^\pm,h_{i1}\}_{i\in \mbI}$. Explicitly, for $s>0$, $x_{is}^\pm$ and $h_{i,s+1}$ are determined by 
\begin{gather*}
 x_{is}^\pm= \pm\frac{1}{2d_i}\left[t_{i1},x_{i,s-1}^\pm\right],\quad \mathrm{ where }\quad t_{i1}=h_{i1}-\frac{\hbar}{2}h_{i0}^2,\label{rec:1}\\
 h_{i,s+1}=[x_{is}^+,x_{i1}^-]. \label{rec:2}
\end{gather*}
\end{lemma}
Let $\{e_i,f_i\}_{i\in \mbI}$ denote the Chevalley generators of $\mfg$, as in \cite{KacBook90}*{\S1.3}, and set 
\begin{equation*}
h_i=d_i\alpha_i^\vee, \quad x_i^+=\sqrt{d_i}e_i, \quad x_i^-=\sqrt{d_i}f_i \quad \forall \; i\in \mbI. 
\end{equation*}
These normalized generators satisfy $(x_i^+,x_i^-)=1$ and $h_i=[x_i^+,x_i^-]$ for all $i\in \mbI$, 
and the relations \eqref{Y:hi=ai}--\eqref{Y:Serre} imply that the assignment
\begin{equation*}
 x_i^\pm \mapsto x_{i0}^\pm, \quad h_i\mapsto h_{i0},\quad h\mapsto h \quad \forall \; i\in \mbI\; \text{ and }\; h\in \mfh,
\end{equation*}
determines a $\C$-algebra homomorphism $U(\mfg)\to \Yhg$. 

\subsection{Generating series and shift automorphisms}\label{ssec:gen-Yhg}

We now spell out a more efficient presentation of $\Yhg$, which can be deduced from \cite{GTL2}*{Prop.~2.3}.
\begin{proposition}\label{P:Y-op}
For each $i\in \mbI$, define $x_i^\pm(u),h_i(u)\in \Yhg[\![u^{-1}]\!]$ by
\begin{equation*}
 x_i^\pm(u)=\sum_{r\geq 0}x_{ir}^\pm u^{-r-1} \quad \text{ and }\quad h_i(u)=\sum_{r\geq 0}h_{ir}u^{-r-1}.
\end{equation*}
Then the defining relations \eqref{Y:hi=ai}--\eqref{Y:Serre} of $\Yhg$ are equivalent to the following relations for $i,j\in \mbI$ and $h,h'\in \mfh$:
\begin{gather}
 h_{i0}=d_i\alpha_i^\vee, \label{Y:hi=ai'}\\
 [h_i(u),h_j(v)]=0, \quad [h,h_j(u)]=0,\quad [h,h']=0,\label{Y:hh'}\\
 [h,x_j^\pm(u)]=\pm \alpha_j(h)x_j^\pm(u), \label{Y:h0x'}\\
 \begin{aligned}\label{Y:xh'}
 (u-v\mp \hbar &d_{ij})h_i(u)x_j^\pm(v)\\
  &=(u-v\pm \hbar d_{ij})x_j^\pm(v)h_i(u)\pm 2d_{ij} x_j^\pm (v)-[h_i(u),x_{j0}^\pm], 
 \end{aligned}\\
 \begin{aligned}\label{Y:xx'}
 (u-v\mp \hbar &d_{ij})x_i^\pm(u)x_j^\pm(v)\\
 &=(u-v\pm \hbar d_{ij})x_j^\pm(v)x_i^\pm(u)+[x_{i0}^\pm,x_j^\pm(v)]-[x_i^\pm(u),x_{j0}^\pm],  
 \end{aligned}\\
 (u-v)[x_i^+(u),x_j^-(v)]=\delta_{ij}(h_i(v)-h_i(u)), \label{Y:xxh'}\\
 \sum_{\pi \in S_{m}} \left[x_i^{\pm}(u_{\pi(1)}), \left[x_i^{\pm}(u_{\pi(2)}), \cdots, \left[x_i^{\pm}(u_{\pi(m)}),x_j^{\pm}(v)\right] \cdots\right]\right] = 0, \label{Y:Serre'}
\end{gather}
where in the last relation $i\neq j$ and $m=1-a_{ij}$.
\end{proposition}
\begin{remark}
Since $x_i^\pm(u),h_i(u)\in u^{-1}\Yhg[\![u^{-1}]\!]$, the relations \eqref{Y:hi=ai'}--\eqref{Y:xxh'} can (and will) be viewed as identities in the algebra $\Yhg[\![u^{-1},v^{-1}]\!]$. Similarly, the Serre relations \eqref{Y:Serre'} should be understood as identities in $\Yhg[\![u_1^{-1},\ldots,u_m^{-1},v^{-1}]\!]$. 
\end{remark}

The Yangian $\Yhg$ admits a family of automorphisms $\{\tau_c\}_{c\in \C}$ defined by
\begin{equation}\label{shift-c}
\begin{gathered}
 \tau_c(h)=h \quad \forall \;h\in \mfh, \\
\tau_c(x_i^\pm(u))=x_i^\pm(u-c), \quad \tau_c(h_i(u))=h_i(u-c) \quad \forall \; i\in \mbI.
\end{gathered}
\end{equation}
This is readily verified using the relations of Proposition \ref{P:Y-op}.
In terms of the generators $x_{ir}^\pm$ and $h_{ir}$, the above formulas read as
\begin{equation*}
 \tau_c(x_{ir}^\pm)=\sum_{k=0}^r \binom{r}{k}x_{ik}^\pm c^{r-k}, \quad \tau_c(h_{ir})=\sum_{k=0}^r \binom{r}{k}h_{ik}c^{r-k} \quad \forall \; i\in \mbI \; \text{ and }\; r\in \N. 
\end{equation*}
Each $\tau_c$ is called a \textit{shift} automorphism. Replacing $c$ by a formal variable $z$, we obtain the formal shift homomorphism
\begin{equation}\label{shift-z}
 \tau_z:\Yhg\into \Yhg[z]
\end{equation}
defined by \eqref{shift-c} with $c$ replaced by $z$. 

\subsection{The Yangian double \texorpdfstring{$\DYhg$}{DYg}}\label{ssec:DYhg}
We now turn to the Yangian double  associated to $\mfg$, as first considered in the work of Khoroshkin--Tolstoy \cite{KT96} in the case where $\mfg$ is finite-dimensional. 
Let $\delta(u)=\sum_{r\in \Z}u^r\in \C[\![u^{\pm1}]\!]$ denote the formal delta function, so that 
\begin{equation*}
 u^{-1}\delta(v/u)=\sum_{r\in \Z} v^r u^{-r-1}\in \C[\![u^{\pm 1},v^{\pm 1}]\!].
\end{equation*}
In what follows, we invoke the standard terminology for topological $\C[\![\hbar]\!]$-algebras; see \cite{KasBook95}*{Def.~XVII.2.2}, for instance. 
\begin{definition}\label{D:DY}
The Yangian double $\DYhg$ is the unital, associative $\C[\![\hbar]\!]$-algebra topologically generated by $h\in \mfh$ and the coefficients $\{\mcX_{ir}^\pm,\mcH_{ir}\}_{i \in \mbI,r\in \Z}$ of the series 
\begin{equation*}
 \mcX_i^\pm(u)=\sum_{r\in \Z}\mcX_{ir}^\pm u^{-r-1} \quad \text{ and }\quad \mcH_i(u)=\sum_{r\in \Z}\mcH_{ir}u^{-r-1},
\end{equation*}
subject to the following relations for all $i,j\in \mbI$ and $h,h'\in \mfh$:
 \begin{gather}
 \mcH_{i0}=d_i\alpha_i^\vee, \label{DY:hi=ai}\\
[\mcH_{i}(u),\mcH_j(v)] =0,\quad [h,\mcH_j(u)]=0,\quad [h,h']=0, \label{DY:hh}\\
[h,\mcX_j^\pm(u)]=\pm \alpha_j(h)\mcX_j^\pm(u), \label{DY:h0x}\\
\left(u-v\mp\hbar d_{ij}\right)\mcH_i(u)\mcX_j^{\pm}(v)=\left(u-v\pm\hbar d_{ij}\right)\mcX_j^{\pm}(v)\mcH_i(u), \label{DY:xh}\\
\left(u-v\mp\hbar d_{ij}\right)\mcX_{i}^{\pm}(u)\mcX_{j}^{\pm}(v) =\left(u-v\pm\hbar d_{ij}\right) \mcX_{j}^{\pm}(v)\mcX_{i}^{\pm}(u) ,  \label{DY:xx}\\
[\mcX_i^+(u),\mcX_j^-(v)] =\delta_{ij} u^{-1}\delta(v/u)\mcH_i(v), \label{DY:xxh}\\
\sum_{\pi \in S_{m}} \left[\mcX_i^{\pm}(u_{\pi(1)}), \left[\mcX_i^{\pm}(u_{\pi(2)}), \cdots, \left[\mcX_i^{\pm}(u_{\pi(m)}),\mcX_j^{\pm}(v)\right] \cdots\right]\right] = 0, \label{DY:Serre}
\end{gather}
where in the last relation $i\neq j$ and $m=1-a_{ij}$.  

The $\C[\hbar]$-form $\DYg$ of $\DYhg$ is defined to be the unital, associative $\C[\hbar]$-algebra generated by $h\in \mfh$ and $\{\mcX_{ir}^\pm,\mcH_{ir}\}_{i \in \mbI,r\in \Z}$, subject to relations \eqref{DY:hi=ai}--\eqref{DY:Serre}. 
\end{definition}
\begin{remark}\label{R:DY} 
\leavevmode
\begin{enumerate}
\item \label{R:DY-1}
The relations \eqref{DY:hh}--\eqref{DY:xxh} are understood to be expanded in the formal series space $\DYhg[\![u^{\pm 1},v^{\pm 1}]\!]$ to yield the corresponding relations for $\DYhg$. Similarly, \eqref{DY:Serre} is to be expanded in $\DYhg[\![u_1^{\pm 1},\ldots,u_m^{\pm 1},v^{\pm 1}]\!]$.
\item\label{R:DY-2}
The above relations are equivalent the relations \eqref{Y:hi=ai}--\eqref{Y:Serre} upon replacing all instances of $x_{ik}^\pm, x_{jk}^\pm, h_{ik}$ and $h_{jk}$ ($k\in \N$) by $\mcX_{ik}^\pm, \mcX_{jk}^\pm, \mcH_{ik}$ and $\mcH_{jk}$, respectively, and allowing $k$ to take arbitrary integer values. 
\end{enumerate}
\end{remark}
Let us now collect some facts about $\DYhg$ and $\DYg$ which follow readily from the above definition. Let $\jmath$ denote the natural $\C[\hbar]$-algebra homomorphism 
\begin{equation*}
\jD:\DYg\to \DYhg. 
\end{equation*}
\begin{proposition}\label{P:D-basic}
\leavevmode
\begin{enumerate}[font=\upshape]
\item\label{D-basic:1} $\jD$ induces an isomorphism of $\C[\![\hbar]\!]$-algebras 
\begin{equation*}
\varprojlim_{n}\left(\DYg/\hbar^n\DYg\right)\iso \DYhg. 
\end{equation*}
\item\label{D-basic:2} For each $i\in \mbI$, we have 
$
[\mcX_{i0}^+,\mcX_i^-(u)]=\mcH_i(u). 
$
Consequently, the set  
\begin{equation*}
\mfh\cup\{\mcX_{ik}^\pm\}_{i\in\mbI, k\in \Z} 
\end{equation*}
generates $\DYg$ as a $\C[\hbar]$-algebra and $\DYhg$ as a topological $\C[\![\hbar]\!]$-algebra. 

\item\label{D-basic:3} $\DYg$ is a $\Z$-graded algebra with $\deg \hbar=1$, $\deg\mfh=0$, and 
\begin{equation*} 
\deg \mcX_{ir}^\pm=\deg \mcH_{ir}=r \quad \forall \;i\in \mbI,\; r\in \Z.
\end{equation*}
\item\label{D-basic:4} The assignment 
\begin{equation*}
x_{ir}^\pm\mapsto \mcX_{ir}^\pm,\quad h_{ir}\mapsto \mcH_{ir}, \quad h\mapsto h\quad \forall \; i\in \mbI,\, r\in \N \; \text{ and }\; h\in \mfh, 
\end{equation*}
extends to a homomorphism of $\Z$-graded $\C[\hbar]$-algebras 
$
\iY:\Yhg\to \DYg. 
$
\end{enumerate}
\end{proposition}
We shall set 
\begin{equation*}
\iYh:=\jD\circ \iY: \Yhg\to \DYhg. 
\end{equation*}
It should be emphasized that, at this point, it is not clear that any of the maps  $\jD$, $\iY$ or $\iYh$ are injective. As a consequence of \eqref{D-basic:1} above, we have 
\begin{equation*}
\Ker(\jD)=\bigcap_{n\in \N}\hbar^n \DYg, 
\end{equation*}
and, as $\DYg$ is not necessarily separated, this ideal need not vanish. 
We will, however, see in Corollary \ref{C:Phi} that both $\iY$ and $\iYh$ are indeed injective. 

\subsection{Translation automorphisms}
We now introduce the so-called \textit{translation automorphisms} of the Yangian double (see \cite{KT96}*{(5.12)}, for instance). These will play a particularly important role in the proof of Theorem \ref{T:hatiso} in Section \ref{sec:Iso}. 
\begin{proposition}\label{P:ti}
Fix $i\in \mbI$. Then the assignment $\mbt_i$ defined by
 \begin{equation*}
  \mbt_{i}(h)=h, \quad \mbt_i(\mcX_{jr}^\pm)=\mcX_{j,r\pm \delta_{ij}}^\pm, \quad \mbt_i(\mcH_{jr})=\mcH_{jr} \quad \forall \; j\in \mbI, \, r\in \Z \; \text{ and }\; h\in \mfh
 \end{equation*}
extends to an automorphism $\mbt_{i}$ of $\DYg$ and of $\DYhg$.
\end{proposition}
\begin{proof}
It suffices to prove the assertion for $\DYg$. 
For each $n\in \Z$, define an assignment $\mbt_i^n$ by
\begin{equation*}
 \mbt_i^n(h)=h,\quad \mbt_i^n(\mcH_j(u))=\mcH_j(u),\quad \mbt_i^n(\mcX_j^\pm(u))=u^{\pm n \delta_{ij}} \mcX_j^\pm(u) \quad \forall \; j\in \mbI,\; h\in \mfh.
\end{equation*}
It is straightforward to verify that $\mbt_i^n$ preserves the relations of Definition \ref{D:DY}. For instance, 
\begin{equation*}
[\mbt_k^n(\mcX_i^+(u)),\mbt_k^n(\mcX_j^-(v))]
=(u/v)^{n\delta_{kj}}\delta_{ij} u^{-1}\delta(v/u)\mcH_{i}(v)=\delta_{ij}u^{-1}\delta(v/u)\mcH_{i}(v),
\end{equation*}
where we have used $(u/v)^n u^{-1}\delta(v/u)=u^{-1}\delta(v/u)$. It follows that $\mbt_i^n$ extends to a $\C[\hbar]$-algebra endomorphism of $\DYg$, which satisfies
$\mbt_i^n=(\mbt_i)^n$ for all $n\in \Z$.  In particular, $\mbt_i$ is an automorphism with inverse $\mbt_i^{-1}$. \qedhere

\end{proof}
%
%

\section{Derived subalgebras and classical limits}\label{sec:Der-CL}

\subsection{The algebras \texorpdfstring{$Y_\hbar\dot{\mathfrak{g}}$}{Ydg} and \texorpdfstring{$\mathrm{D}Y_\hbar\dot{\mathfrak{g}}$}{DYdg}}

In the current literature on Yangians of infinite-dimensional Kac--Moody algebras, both the full Yangian $\Yhg$ of Definition \ref{D:Y} and the Yangian 
$\Yhdg$ of the derived Lie subalgebra $\dot\mfg\subset \mfg$, defined below, have independently been considered; see \cites{GNW,GRWvrep,YaGuPBW}, for instance. 

The results of this paper, which are primarily stated for $\Yhg$ and $\DYhg$, are entirely valid for $\Yhdg$ and $\DYhdg$. In this subsection, we make this transparent by clarifying the precise relationship between $\Yhg$ and $\Yhdg$, and $\DYhg$ and $\DYhdg$. 
\begin{definition}
The Yangian $\Yhdg$ is the unital, associative $\C[\hbar]$-algebra generated by $\{x_{ir}^\pm,h_{ir}\}_{i\in\mbI,r\in \N}$, subject to the relations \eqref{Y:xxh} - \eqref{Y:Serre} of Definition \ref{D:Y}, in addition to 
\begin{equation*}
[h_{ir},h_{js}]=0,\quad [h_{i0},x_{js}^\pm]=\pm 2d_{ij} x_{js}^\pm \quad \forall \; i,j\in \mbI,\; r,s\in \N.
\end{equation*}
\end{definition}
We first observe that $\Yhdg$ admits the structure of a $Q$-graded $\C[\hbar]$-algebra
\begin{equation*}
\Yhdg=\bigoplus_{\beta\in Q}\Yhdg_\beta,\\
\end{equation*}
determined by assigning $\deg h_{ir}=0$ and $\deg x_{ir}^\pm=\pm \alpha_i$ for all $i\in \mbI$ and $r\in \N$. Next, let us fix a decomposition of (abelian) Lie algebras
\begin{equation*}
\mfh=\dot\mfh\oplus\ddot\mfh, \quad \text{ where }\quad \dot\mfh=\bigoplus_{i\in \mbI}\C\alpha_i^\vee.
\end{equation*}
The Lie algebra $\ddot\mfh$ then acts on $\Yhdg$ by the commuting $\C[\hbar]$-linear derivations
uniquely determined by 
\begin{equation}\label{h-der}
h\cdot x_\beta =\beta(h)x_\beta \quad \forall \; h\in \ddot\mfh,\, x_\beta\in \Yhdg_\beta. 
\end{equation}
We can thus form the \textit{crossed product} (or smash product) algebra $\Yhdg\rtimes U(\ddot\mfh)$ over the complex numbers \cite{Mont}*{Def.~4.1.3}. As a vector space, we have 
\begin{equation*}
\Yhdg\rtimes U(\ddot\mfh)=\Yhdg\otimes_\C U(\ddot\mfh),
\end{equation*} 
with associative multiplication $\bullet$ defined on simple tensors by
\begin{equation*}
(x\otimes h)\bullet (y\otimes h^\prime)=x(h_1\cdot y)\otimes h_2 h^\prime \quad \forall\; x,y\in \Yhdg\; \text{ and }\; h,h^\prime\in U(\ddot\mfh),
\end{equation*}
where we have used the sumless Sweedler notation $\Delta(h)=h_1\otimes h_2$ for the standard coproduct on $U(\ddot\mfh)$. As the underlying action of $U(\ddot\mfh)$
is $\C[\hbar]$-linear, this defines a $\C[\hbar]$-algebra structure on $\Yhdg\rtimes U(\ddot\mfh)$. 
We then have the following result. 
\begin{proposition}\label{P:Yhdg-x-ddh}
The assignment 
\begin{equation*}
h\mapsto 1\otimes h, \quad x_{ir}^\pm \mapsto x_{ir}^\pm\otimes 1, \quad h_{ir}\mapsto h_{ir}\otimes 1,
%
\end{equation*}
for $h\in \ddot\mfh$, $i\in \mbI$ and $r\in \N$, uniquely extends to an isomorphism of $\C[\hbar]$-algebras 
\begin{equation*}
\Yhg\iso \Yhdg \rtimes U(\ddot\mfh). 
\end{equation*}
\end{proposition}
The proof of the proposition is entirely analogous to the argument that $U(\mfg)$ itself decomposes as 
$U(\mfg)\cong U(\dot\mfg)\rtimes U(\ddot\mfh)$, and is therefore omitted.  

A nearly identical story unfolds if $\Yhg$ is replaced by $\DYhg$. The only subtlety which arises is that the crossed product construction should be carried out in the category of topological $\C[\![\hbar]\!]$-modules. We summarize these results below, beginning with the definition 
of $\DYhdg$.
\begin{definition}\label{D:DYdotg}
The Yangian double $\DYhdg$ is the unital, associative $\C[\![\hbar]\!]$-algebra topologically generated by $\{\mcX_{ir}^\pm,\mcH_{ir}\}_{i \in \mbI,r\in \Z}$,
subject to the relations \eqref{DY:xh} - \eqref{DY:Serre} of Definition \ref{D:DY}, in addition to 
\begin{equation*}
[\mcH_{ir},\mcH_{js}]=0,\quad [\mcH_{i0},\mcX_{js}^\pm]=\pm 2 d_{ij} \mcX_{js}^\pm \quad \forall \; i,j\in \mbI,\; r,s\in \Z.
\end{equation*}

The $\C[\hbar]$-form $\DYdg$ of $\DYhdg$ is the unital, associative $\C[\hbar]$-algebra generated by 
$\{\mcX_{ir}^\pm,\mcH_{ir}\}_{i \in \mbI,r\in \Z}$, subject to the same set of relations. 
\end{definition}
The algebra $\DYdg$ is itself $Q$-graded with $\deg \mcH_{ir}=0$ and $\deg \mcX_{ir}^\pm=\pm \alpha_i$:
\begin{equation*}
\DYdg=\bigoplus_{\beta\in Q}\DYdg_\beta.
\end{equation*}
As in the Yangian case, we have an action of the Lie algebra $\ddot\mfh$ on $\DYdg$ by derivations, uniquely determined by \eqref{h-der}, where  
$x_\beta$ now takes values in $\DYdg_\beta$. 
Each such derivation is $\C[\hbar]$-linear, and therefore determines a $\C[\![\hbar]\!]$-linear derivation of the algebra
\begin{equation*}
\DYhdg\cong \varprojlim_n(\DYdg/\hbar^n\DYdg). 
\end{equation*}

We thus have an action of $\ddot\mfh$ on $\DYhdg$ by derivations, and may form the spaces 
\begin{equation*}
\DYdg\rtimes U(\ddot\mfh)\quad \text{ and }\quad \DYhdg\rtimes U(\ddot\mfh),
\end{equation*}
 which are naturally algebras over $\C[\hbar]$ and $\C[\![\hbar]\!]$, respectively. We then have the following analogue of Proposition \ref{P:Yhdg-x-ddh}.  
\begin{proposition}\label{P:Dhdg-x-ddh}
The assignment 
\begin{equation*}
h\mapsto 1\otimes h, \quad \mcX_{ir}^\pm \mapsto \mcX_{ir}^\pm\otimes 1, \quad \mcH_{ir}\mapsto \mcH_{ir}\otimes 1,
\end{equation*}
for $h\in \ddot\mfh$, $i\in \mbI$ and $r\in \Z$, uniquely extends to yield algebra isomorphisms 
\begin{equation*}
\DYg\iso \DYdg \rtimes U(\ddot\mfh) \quad \text{ and }\quad \DYhg\iso \DYhdg \rtimes_\hbar U(\ddot\mfh),
\end{equation*}
where $\DYhdg \rtimes_\hbar U(\ddot\mfh)$ is the $\hbar$-adic completion of  $\DYhdg \rtimes U(\ddot\mfh)$.
\end{proposition}
%
%

\subsection{The classical limits \texorpdfstring{$\mfs$}{s} and \texorpdfstring{$\mft$}{t}}
Modulo the ideal generated by $\hbar$, the defining relations of $\Yhdg$ and $\DYhdg$ are of Lie type. It follows that $\Yhdg$ and $\DYhdg$ deform the enveloping algebras of certain infinite-dimensional complex Lie algebras $\mfs$ and $\mft$, respectively. 
%
%
In this section, we overview the abstract definitions of $\mfs$ and $\mft$, together with their $\Yhg$ and $\DYhg$ counterparts.

Henceforth, the symbol $\mfa$ is understood to take value $\mfs$ or $\mft$, and we set
\begin{equation*}
\Z_\mfs=\N \quad \text{ and }\quad \Z_\mft=\Z. 
\end{equation*}
\begin{definition}
The complex Lie algebra $\mfa$ is defined to be the quotient of the free Lie algebra on  $\{X_{ir}^\pm, H_{ir}\}_{i\in \mbI,r\in \Z_\mfa}$ by the ideal generated by the following relations, for $i,j\in \mbI$ and $r,s\in \Z_\mfa$: 
\begin{gather}
  [H_{ir},H_{js}]=0, \label{s-HH}\\
  [H_{ir},X_{js}^\pm]=\pm 2d_{ij}X_{j,r+s}^\pm, \label{s-HX}\\
  [X_{ir}^+,X_{js}^-]=\delta_{ij}H_{i,r+s}, \label{s-X+X-}\\
  [X_{i,r+1}^\pm,X_{js}^\pm]=[X_{ir}^\pm,X_{j,s+1}^\pm], \label{s-xx}\\
  \mathrm{ad} (X_{i0}^\pm)^{1-a_{ij}}(X_{js}^\pm)=0 \quad \text{ for }\; i\neq j. \label{s-serre}
\end{gather}
\end{definition}
We note that $\mfa$ is a $\Z_\mfa$-graded Lie algebra with 
$\deg X_{ir}^\pm=\deg H_{ir}=r$ for all  $i\in \mbI$ and  $r\in \Z_{\mfa}$.
Additionally, the assignment
\begin{equation*}
X_{ir}^\pm \mapsto x_i^\pm\otimes t^r,\quad  H_{ir}\mapsto h_i\otimes t^r \quad \forall \; i\in \mbI,\; r\in \Z_\mfa, 
\end{equation*}
extends to yield graded epimorphisms 
\begin{equation}\label{pi_a}
\pi_\mfs:\mfs\onto \dot\mfg[t] \quad \text{ and }\quad \pi_\mft:\mft\onto \dot\mfg[t^{\pm 1}], 
\end{equation}
which are isomorphisms when $\mfg$ is finite-dimensional. When $\mfg$ is an untwisted affine Lie algebra with underlying simple Lie algebra $\bar\mfg\ncong \mfsl_2$, one has instead
\begin{equation}\label{st-uce}
\mfs\cong\mathfrak{uce}(\dot\mfg[t]) \quad \text{ and }\quad \mft_\kappa\cong\mathfrak{uce}(\dot\mfg[t^{\pm 1}]),
\end{equation}
where $\mathfrak{uce}(\mfp)$ denotes the universal central extension of a perfect Lie algebra $\mfp$, 
and $\mft_\kappa$ is a one-dimensional central extension of $\mft$, defined for $\mfg$ of any type, constructed as follows. 
Define a linear map $\bar\kappa:\dot\mfg[t^{\pm 1}]\otimes \dot\mfg[t^{\pm 1}]\to \C$ by
\begin{equation*}
 \bar\kappa(f(t),g(t))=\mathrm{Res}_{t}\left(\partial_t(f(t)),g(t)\right) \quad \forall \; f(t),g(t)\in \dot\mfg[t^{\pm 1}],
\end{equation*}
where the invariant form $(\,,\,)|_{\dot\mfg\times\dot\mfg}$ has been naturally extended to a bilinear form on $\dot\mfg[t^{\pm 1}]\otimes \dot\mfg[t^{\pm 1}]$ with values in $\C[t^{\pm 1}]$, $\partial_t:\dot\mfg[t^{\pm 1}]\to \dot\mfg[t^{\pm 1}]$ is the formal derivative operator, and $\mathrm{Res}_t:\C[t^{\pm 1}]\to \C$ is the formal residue.  
One verifies as in \cite{KacBook90}*{\S7.2} that $\bar\kappa$ is a $\C$-valued $2$-cocycle on $\dot\mfg[t^{\pm 1}]$. It follows that 
\begin{equation*}
 \kappa= \bar\kappa\circ \pi_\mft^{\otimes 2}:\mft\otimes \mft\to \C
\end{equation*}
is a $\C$-valued $2$-cocycle on $\mft$. 
\begin{definition}\label{D:t-kappa}
The Lie algebra $\mft_\kappa$ is the central extension of $\mft$ by the cocycle $\kappa$.  That is,
$\mft_\kappa=\mft\oplus \C \mathrm{K}$ as a vector space, with Lie bracket given by $[\mft,\mathrm{K}]=0$ and 
\begin{equation*}
 [x,y]=[x,y]_{\mft}+\kappa(x,y)\mathrm{K} \quad \forall \; x,y\in \mft.
\end{equation*}
\end{definition}
The assertion of \eqref{st-uce} is non-trivial, and has been established in the work of Moody, Rao and Yokonuma \cite{MRY90}. These isomorphisms appear in the form stated above in \cite{GRWvrep}, where $\mft_\kappa$ is itself denoted $\mft$. A deeper analysis of these results will be given in the course of the proof of Theorem \ref{T:PBW} in Section \ref{ssec:aff-II}.  

Returning to our general discussion of $\mfa$, note that the assignment $\deg H_{ir}=0$ and $\deg X_{ir}^\pm=\pm \alpha_i$, for all $i\in \mbI$ and $r\in \Z_\mfa$, defines a $Q$-grading 
\begin{equation*}
\mfa=\bigoplus_{\beta\in Q}\mfa_\beta.
\end{equation*} 
The commutative Lie algebra $\ddot\mfh$ acts on $\mfa$ by the derivations uniquely determined by \eqref{h-der} with $x_\beta\in\mfa_\beta$. We may therefore take the semidirect product of Lie algebras
\begin{equation*}
\mfa\rtimes \ddot\mfh.
\end{equation*}
We then have the following result, where the notation $x_{ir}^{\mfa,\pm},h_{ir}^{\mfa}$ is used to denote $x_{ir}^\pm,h_{ir}\in \Yhdg$ if $\mfa=\mfs$, and $\mcX_{ir}^\pm,\mcH_{ir}\in \DYhdg$ if $\mfa=\mft$.
\begin{proposition}\label{P:red-h}
The assignment 
\begin{equation*}
X_{ir}^\pm \mapsto x_{ir}^{\mfa,\pm}\mod \hbar, \quad H_{ir}\mapsto h_{ir}^{\mfa}\mod \hbar \quad \forall \; i\in \mbI,\; r\in \Z_{\mfa}
\end{equation*}
 uniquely extends to isomorphisms of graded algebras
\begin{equation*}
U(\mfs)\iso\Yhdg/\hbar\Yhdg \quad \text{ and }\quad U(\mft)\iso \DYhdg/\hbar\DYhdg.
\end{equation*}
Tensoring with the identity $\id$ on $U(\ddot\mfh)$ yields isomorphisms
\begin{equation*}
U(\mfs\rtimes \ddot\mfh)\iso\Yhg/\hbar\Yhg \quad \text{ and }\quad U(\mft\rtimes \ddot\mfh)\iso \DYhg/\hbar\DYhg.
\end{equation*}
\end{proposition}
The first assertion of the proposition, for $\Yhdg$, is precisely \cite{GRWvrep}*{Prop.~2.6}.
The $\DYhdg$ analogue of this result follows from an identical argument (see also \cite{GRWvrep}*{Prop.~3.6}).
The second part of the proposition is then a consequence of Propositions \ref{P:Yhdg-x-ddh} and \ref{P:Dhdg-x-ddh}, which imply there are algebra isomorphisms
\begin{gather*}
\Yhg/\hbar\Yhg
\cong 
\Yhdg/\hbar\Yhdg \rtimes U(\ddot\mfh)
\cong 
U(\mfs\rtimes\ddot\mfh),\\
\DYhg/\hbar\DYhg
\cong 
\DYhdg/\hbar\DYhdg \rtimes U(\ddot\mfh)
\cong 
U(\mft\rtimes\ddot\mfh),
\end{gather*}
where we have employed the fact that $U(\mfa\rtimes \ddot\mfh)\cong U(\mfa)\rtimes U(\ddot\mfh)$.

\section{Extending the shift automorphism}\label{sec:Phi}

The primary goal of this section is to introduce the formal shift operator $\Phi_z$ together with its evaluation $\Phi_c$ at any invertible complex number $c\in \C^\times$. This will be achieved in Theorem \ref{T:Phi}, after first proving a collection of preliminary results on completed Yangians and formal series algebras. 
We will then conclude this section by spelling out a number of direct consequences to Theorem \ref{T:Phi} in Sections \ref{ssec:Phi-imp} and \ref{ssec:gr-tw}.

\subsection{Completed Yangian}\label{ssec:CY}
Let $\cYhg$ denote the completion of $\Yhg$ with respect to its $\N$-grading:
\begin{equation*}
 \cYhg=\prod_{k\in \N}\Yhg_k. 
\end{equation*}
Since $\hbar$ has degree one, $\cYhg$ is a unital, associative $\C[\![\hbar]\!]$-algebra. 
Consider now the ideal $\Yhgp\subset \Yhg$ generated by elements of strictly positive degree: 
\begin{equation*}
\Yhgp=\bigoplus_{k>0}\Yhg_k.
\end{equation*}
\begin{lemma}\label{L:whYhg} $\cYhg$ admits the following properties. 
\leavevmode
\begin{enumerate}[font=\upshape]
\item\label{whYhg:1}  The canonical $\C[\hbar]$-algebra homomorphism 
$
\Yhg\to \varprojlim_{n}\!\left(\Yhg/\Yhgp^n \right)
$
extends to an isomorphism of $\C[\![\hbar]\!]$-algebras 
\begin{equation*}
\cYhg\iso \varprojlim_{n}\left(\Yhg/\Yhgp^n \right)\!.
\end{equation*}
\item\label{whYhg:2} $\cYhg$ separated and complete as a $\C[\![\hbar]\!]$-module,
\item\label{whYhg:3} $\cYhg$ is a torsion free $\C[\![\hbar]\!]$-module, provided $\Yhg$ is a torsion free $\C[\hbar]$-module.
\end{enumerate}

\end{lemma}
\begin{proof}
By Lemma \ref{L:deg1}, $\Yhg_0$ and $\Yhg_1$ generate $\Yhg$ as a $\C[\hbar]$-algebra, and consequently, we have 
\begin{equation*}
\Yhgp^n=\bigoplus_{k\geq n} \Yhg_k\quad \text{ for all }\; n\in \N.
\end{equation*}
Lemma \ref{L:whYhg} thus follows from Proposition \ref{P:A-comp} of Appendix \ref{App:A} with the algebra $\mathrm{A}$ taken to be $\Yhg$. Parts \eqref{whYhg:2} and \eqref{whYhg:3} may also be proven as in \cite{GTL1}*{Prop.~6.3}. \qedhere
\end{proof}
%
%


\subsection{Formal series spaces}\label{ssec:FSS}

Let $\Yhg[z;z^{-1}]\!]$ denote the algebra of formal Laurent series in $z^{-1}$ with ceofficients in $\Yhg$:
\begin{equation*}
\Yhg[z;z^{-1}]\!]=\bigcup_{n\in \N}z^n \Yhg[\![z^{-1}]\!]\subset \Yhg[\![z^{\pm 1}]\!].
\end{equation*} 
Define $\Yhgz\subset \Yhg[\![z^{-1}]\!]$ by 
\begin{equation*}
\Yhgz=\prod_{k\in \N}\Yhg_k z^{-k}. 
\end{equation*}
 Let $\LzYhg$ denote the $\C[z^{\pm 1}]$-submodule of $\Yhg[z;z^{-1}]\!]$ generated by $\Yhgz$. 
 The following proposition outputs a set of valuable properties characterizing this space and its 
 $\hbar$-adic completion. 
\begin{proposition}\label{P:whYgz}
Let $v$ be an indeterminate and equip $\cYg{v}[z^{\pm 1}]$ with the $\C[\hbar]$-algebra structure determined by $\hbar\cdot 1=v z$. Then: 
\begin{enumerate}[font=\upshape]
\item\label{whYgz:1} $\LzYhg$ is a $\Z$-graded $\C[\hbar]$-algebra with $\LzYgr{k}=z^k \Yhgz$.  In particular, 
\begin{equation*}
\LzYhg=\bigoplus_{n\in \Z} z^n \Yhgz. 
\end{equation*}
\item\label{whYgz:2} The graded linear map $\LzYhg\to \cYg{v}[z^{\pm 1}]$ given by
\begin{equation*}
z^n f_\hbar(z)\mapsto z^n f_v(1) \quad \forall \quad f_\hbar(z)\in \Yhgz,\; n\in \Z,
\end{equation*}
is an isomorphism of graded $\C[\hbar]$-algebras. 

\item\label{whYgz:3} The $\hbar$-adic completion $\LzhYhg$ of $\LzYhg$ is the subspace of $\cYhg[\![z^{\pm 1}]\!]$ consisting of formal series 
\begin{equation*}
\sum_{k\in \Z} z^k f_k(z), \quad f_k(z)\in \Yhgz
\end{equation*}
with the property that, for each $m\in \N$, $\exists$ $N_m\in \N$ such that 
\begin{equation*}
f_k(z)\in (\hbar/z)^m \Yhgz \quad \forall \quad |k|\geq N_m.
\end{equation*}
\item\label{whYgz:4} For each $c\in \C^\times$, the map 
\begin{equation*}
\mathscr{Ev}_c:\LzhYhg\to \cYhg,\quad f(z)\mapsto f(c),
\end{equation*}
is an epimorphism of $\C[\![\hbar]\!]$-algebras.
\end{enumerate}
\end{proposition}
\begin{proof}
As $\Yhgz$ is a $\C$-algebra, $\hbar\in z\Yhgz$ and 
\begin{equation*}
z^n \Yhgz \cdot z^m \Yhgz\subset   z^{n+m} \Yhgz \quad \forall\; n,m\in \Z,
\end{equation*}
$\LzYhg$ is a $\C[\hbar]$-algebra, which will be $\Z$-graded provided the sum $\sum_{n\in \Z} z^n \Yhgz$ is direct. This assertion is readily verified, and hence Part \eqref{whYgz:1} holds.

Part \eqref{whYgz:2} is a consequence of Part \eqref{whYgz:1}, the definition of the $\C[\hbar]$-module structure on  $\cYg{v}[z^{\pm 1}]$, and the fact that 
\begin{equation*}
\Yhgz\to \cYhg, \quad f(z)\mapsto f(1),
\end{equation*}
is an isomorphism of $\C$-algebras.

Consider now Part \eqref{whYgz:3}. Since $z\in \cYg{v}[z^{\pm 1}]$ is a unit, Part \eqref{whYgz:2} yields 
\begin{equation*}
\LzhYhg
=\varprojlim_{n}\left(\LzYhg/\hbar^n\LzYhg\right)
\cong 
\varprojlim_{n}\left(\cYg{v}[z^{\pm 1}]/v^n\cYg{v}[z^{\pm 1}]\right).
\end{equation*}
Part \eqref{whYgz:3} thus follows from the identification of Part \eqref{whYhg:2},  Lemma \ref{L:whYhg}, and the following straightforward general result: 

If $\mathrm{A}$ is a separated and complete $\C[\![\hbar]\!]$-module, then the $\hbar$-adic completion of $\mathrm{A}[z^{\pm 1}]$ is equal to the subspace of $\mathrm{A}[\![z^{\pm 1}]\!]$ consisting of all series 
\begin{equation*}
\sum_{k\in \Z}x_kz^{k} \in \mathrm{A}[\![z^{\pm 1}]\!]
\end{equation*}
satisfying the condition that, for each $m\in \N$, $\exists$ $N_m\in \N$ such that 
\begin{equation*}
x_k\in \hbar^m \mathrm{A} \quad \forall \quad |k|\geq N_m.
\end{equation*}

Let us now turn to Part \eqref{whYgz:4}. Composing the algebra epimorphism $\cYg{v}[z^{\pm 1}]\onto \cYhg$, $f_v(z)\mapsto f_\hbar(c)$ with the isomorphism of \eqref{whYgz:2}, we obtain an epimorphism of $\C[\hbar]$-algebras 
\begin{equation*}
\mathscr{Ev}_c^\prime:\LzYhg\onto \cYhg,  
\end{equation*}
given by evaluating $z\mapsto c$. Since, by Lemma \ref{L:whYhg}, $\cYhg$ is separated and complete, $\mathscr{Ev}_c^\prime$ induces $\mathscr{Ev}_c$ as in the statement of the proposition. \qedhere
\end{proof}
%
%
\subsection{The formal shift operator \texorpdfstring{$\Phi_z$}{Phi\_z}}\label{ssec:Phi_z}
Let $\tau_c$ and $\tau_z$ be the shift homomorphisms of \eqref{shift-c} and \eqref{shift-z}, respectively, and recall that $\jD$ and $\iYh$ are the natural homomorphisms
\begin{equation*}
\jD:\DYg\to \DYhg \quad \text{ and }\quad \iYh:\Yhg\to \DYhg
\end{equation*}
introduced in Section \ref{ssec:DYhg}. 
In addition, we shall set
\begin{equation*}
\partial_z^{(n)}=\frac{1}{n!}(\partial_z)^n \quad \forall \quad n\in \N,
\end{equation*} 
where $\partial_z$ is the formal derivative operator with respect to $z$. With the machinery of Sections \ref{ssec:CY} and \ref{ssec:FSS} at our disposal, we are now prepared to state and prove our first main result.
\begin{theorem}\label{T:Phi}
\leavevmode
\begin{enumerate}[font=\upshape]
\item\label{Phi:1} There is a unique homomorphism of $\C[\![\hbar]\!]$-algebras 
\begin{equation*}
\Phi_z:\DYhg\to \LzhYhg
\end{equation*}
with the property that $\Phi_z\circ \iYh=\tau_z$. It is given by 
\begin{equation}\label{Phi_z}
 \begin{gathered}
 \Phi_{z}(h)=h \quad \forall \; h\in \mfh, \\
 u\Phi_{z}(\mcH_i(u))=\sum_{n\in \N}h_{i,n}\partial_z^{(n)}\!\left(\delta(z/u)\right),
\quad u\Phi_{z}(\mcX_i^\pm(u))=\sum_{n\in \N}x_{i,n}^\pm \partial_z^{(n)}\!\left(\delta(z/u)\right),
  \end{gathered}
 \end{equation}
 for all $i\in \mbI$. 
 \item\label{Phi:2} The compostion $\Phi_z\circ \jD$ is a $\Z$-graded $\C[\hbar]$-algebra homomorphism
 \begin{equation*}
 \Phi_z\circ \jD: \DYg\to \LzYhg=\bigoplus_{n\in \Z}z^n \Yhgz
 \subset \Yhg[z;z^{-1}]\!]. 
 \end{equation*}
 \item\label{Phi:3} Fix $c\in \C^\times$. Then $\Phi_c=\mathscr{Ev}_c\circ \Phi_z$ is the unique homomorphism of $\C[\![\hbar]\!]$-algebras 
\begin{equation*}
\Phi_c:\DYhg\to \cYhg
\end{equation*}
satisfying $\Phi_c\circ \iYh=\tau_c$. 
\end{enumerate}
\end{theorem}
The proof of the theorem will be given in \S\ref{ssec:PfofPhi} below. Let us first examine the formula \eqref{Phi_z} in more detail. 
 Expanding the formal delta function $u^{-1}\delta(z/u)$ as 
\begin{equation*}
u^{-1}\delta(z/u)=\exp(-z \partial_u)(1/u)+ \exp(-u \partial_z)(1/z),
\end{equation*}
we find that 
\begin{align*}
\partial_z^{(n)}\!\left(u^{-1}\delta(z/u)\right)
%
%
%
&=\exp (-z \partial_u)(u^{-n-1})+(-1)^n\exp(-u \partial_z)(z^{-n-1}).
\end{align*}
Hence, from the second line of \eqref{Phi_z}, we obtain
\begin{equation}\label{Phi_z-tau_z}
\begin{gathered}
(\Phi_z \circ \iYh) h_i(u)=\exp (-z \partial_u)h_i(u)=h_i(u-z),\\
(\Phi_z \circ \iYh) x_i^\pm(u)
=x_i^\pm(u-z),
\end{gathered}
\end{equation}
for all $i\in \mbI$. We may thus conclude that $\Phi_z\circ\iYh=\tau_z$ will hold, provided that $\Phi_z$, as given by \eqref{Phi_z}, is an algebra homomorphism. 

The above expansion of $\partial_z^{(n)}\!\left(u^{-1}\delta(z/u)\right)$ also implies that
\begin{equation}\label{vertex}
\Phi_z(\mcX_{i,-n-1}^\pm)=(-1)^{n+1}\partial_z^{(n)} x_i^\pm(-z) \quad \forall \; i\in \mbI, \, n\in \N. 
\end{equation} 
In particular, since $x_i^\pm(z)\in z^{-1}\Yhgz$ and $\partial_z^{(n)}$ is a degree $-n$ operator on $\LzYhg$, 
\begin{equation*}
(-1)^{n+1}\partial_z^{(n)} x_i^\pm(-z)\in z^{-n-1}\Yhgz\subset \LzYhg.
\end{equation*}
Consequently, Part \eqref{Phi:2} of the theorem will follow automatically from Part \eqref{Phi:1} and the second statement of Proposition \ref{P:D-basic}.

\subsection{Proof of Theorem \ref{T:Phi}}\label{ssec:PfofPhi}
Let us begin by establishing that there is at most one homomorphism $\Phi_z:\DYhg\to \LzhYhg$ with the property that $\Phi_z\circ\iYh=\tau_z$. Our argument will also imply the uniqueness of $\Phi_c$, as in the statement of the theorem.

\noindent\textit{Proof of uniqueness.} 
Let $\Phi_z$ be such a homomorphism, and fix $i\in \mbI$. Our starting point is the relation 
\begin{equation}\label{Ti-DY}
[\iYh(t_{i1}),\mcX_{is}^\pm]=\pm 2 d_i \mcX_{i,s+1}^\pm \quad \forall\; s\in \Z,
\end{equation}
where $t_{i1}=h_{i1}-\frac{\hbar}{2}h_{i0}^2$, as in Lemma \ref{L:deg1}. This relation is proven in the same way as its $\Yhg$-counterpart; see \eqref{R:DY-2} of Remark \ref{R:DY} and Lemma \ref{L:deg1}. It implies that 
\begin{equation*}
\mathrm{ad}(\mathrm{T}_i)^k(\mcX_{is}^\pm)=(\pm 1)^k\mcX_{i,s+k}^\pm \quad \forall \; s\in \Z,\;k\in \N, \quad \text{ where }\;  \mathrm{T}_i=\frac{\iota(t_{i1})}{2d_i}.
\end{equation*}
Applying $\Phi_z$, and using that $\tau_z(t_{i1})=t_{i1}+zh_{i0}$, we obtain 
\begin{equation*}
\mathrm{ad}\!\left(\mathrm{T}_i+ \frac{z}{2d_i}\mcH_{i0}\right)^{\!k}\Phi_z(\mcX_{is}^\pm)=(\pm 1)^k\Phi_z(\mcX_{i,s+k}^\pm)  \quad \forall \; s\in \Z,\;k\in \N.
\end{equation*}
By \eqref{DY:hi=ai} and \eqref{DY:h0x}, 
$
[\mcH_{i0},\Phi_z(\mcX_{is}^\pm)]=\pm 2d_i \Phi_z(\mcX_{is}^\pm).
$
It follows that the above is equivalent to 
\begin{equation*}
(z\pm \mathrm{ad}(\mathrm{T}_i))^k\Phi_z(\mcX_{is}^\pm)=\Phi_z(\mcX_{i,s+k}^\pm)  \quad \forall \; s\in \Z,\;k\in \N.  
\end{equation*}
Fixing $n\in \N$ and taking $k=n+1$ and $s=-n-1$, we deduce that  
\begin{equation}\label{Phi-invert}
(z\pm \mathrm{ad}(\mathrm{T}_i))^{n+1}\Phi_z(\mcX_{i,-n-1}^\pm)=\mcX_{i0}^\pm. 
\end{equation}
As $\mathrm{T}_i\in \Yhg_1$,  
\begin{equation}\label{Phi-invert:2}
\left(z\pm \mathrm{ad}(\mathrm{T}_i)\right)^{-n-1}=\sum_{p\geq 0}(-1)^n\mathrm{ad}(\mp\mathrm{T}_i)^p\partial_z^{(n)}(z^{-p-1})
\end{equation}
is a $\C[\![\hbar]\!]$-linear endomorphism of $\LzhYhg$. Applying it to \eqref{Phi-invert} and employing \eqref{Ti-DY}, we recover \eqref{vertex}: 
\begin{equation*}
\Phi_z(\mcX_{i,-n-1}^\pm)=\sum_{p\geq 0}(-1)^{n+p}x_{ip}^\pm\partial_z^{(n)}(z^{-p-1})
=(-1)^{n+1}\partial_z^{(n)} x_i^\pm(-z).
\end{equation*}
By Part \eqref{D-basic:2} of Proposition \ref{P:D-basic}, this identity, together with the requirement $\Phi_z\circ\iYh=\tau_z$, completely determines $\Phi_z$. This proves the uniqueness of $\Phi_z$.

Observe that, since the evaluation of \eqref{Phi-invert:2} at $z=c\in \C^\times$ defines an honest $\C[\![\hbar]\!]$-linear endormorphism of $\cYhg$, the above uniqueness argument is completely valid with $z$ replaced by a fixed scalar $c\in \C^\times$. It thus proves that there is at most one $\C[\![\hbar]\!]$-algebra homomorphism $\Phi_c:\DYhg\to \cYhg$ such that $\Phi_c\circ\iYh = \tau_c$. 

\noindent\textit{Proof of \eqref{Phi:1} and \eqref{Phi:2}.} Next, we prove that the assignment $\Phi_z$ defined by \eqref{Phi_z} preserves the defining relations of $\DYhg$. Since $\LzhYhg$ is separated and complete (by \eqref{whYgz:3} of Proposition \ref{P:whYgz}), this will imply that \eqref{Phi_z} indeed extends to a homomorphism of $\C[\![\hbar]\!]$-algebras 
\begin{equation*}
\Phi_z:\DYhg\to \LzhYhg,
\end{equation*}
which, by the remarks following the statement of the theorem, will complete the proof of both Parts \eqref{Phi:1} and \eqref{Phi:2} of the theorem.

It is clear from \eqref{Y:hi=ai} and \eqref{Y:hh} that $\Phi_z$
preserves the relations \eqref{DY:hi=ai} and \eqref{DY:hh}. 

\medskip

\noindent \textit{The relation \eqref{DY:h0x}.} By \eqref{Y:h0x}, for each $h\in \mfh$ and $j\in \mbI$ we have
\begin{equation*}
 [\Phi_z(h),\Phi_z(\mcX_j^\pm(u))]=\sum_{n\in \mathbb{N}}[h,x_{j,n}^\pm]\partial_z^{(n)}\!\left(u^{-1}\delta(z/u)\right)
 =\pm \alpha_j(h)\Phi_z(\mcX_j^\pm(u)).
\end{equation*}

\medskip 

\noindent \textit{The relations \eqref{DY:xh} and \eqref{DY:xx}.} For each $n\in \mathbb{N}$, we have 
\begin{equation}\label{delta_z}
\begin{gathered}
 (u-v)\partial_z^{(n)}\!\left(u^{-1}\delta(z/u)\right)=\partial_z^{(n-1)}\!\left(u^{-1}\delta(z/u)\right)+(z-v)\partial_z^{(n)}\!\left(u^{-1}\delta(z/u)\right),\\
 (u-z)\partial_z^{(n)}\!\left(u^{-1}\delta(z/u)\right)=\partial_z^{(n-1)}\!\left(u^{-1}\delta(z/u)\right),
 \end{gathered}
\end{equation}
where $\partial_z^{(n-1)}\!\left(u^{-1}\delta(z/u)\right)=0$ if $n=0$. The second relation is obtained from the first by setting $v=z$, and the first relation is proven by induction on $n$.

For each $n,m\in \mathbb{N}$, set 
\begin{equation*}
 f_{n,m}(u,v,z)=\partial_z^{(n)}\!\left(u^{-1}\delta(z/u)\right) \partial_z^{(m)}\!\left(v^{-1}\delta(z/v)\right). 
\end{equation*}
Then \eqref{delta_z} implies that $f_{n,m}(u,v,z)$ satisfies 
\begin{equation}\label{delta_z'}
 (u-v)f_{n,m}(u,v,z)=f_{n-1,m}(u,v,z)-f_{n,m-1}(u,v,z),
\end{equation}
where $f_{-1,m}(u,v,z)=f_{n,-1}(u,v,z)=0$. 

We now apply this to prove that $\Phi_z$ preserves \eqref{DY:xh} and \eqref{DY:xx}. Fix $i,j\in \mbI$, and let $(\mcY_i(u),y_{ir})$
denote $(\mcX_i^\pm(u),x_{ir}^\pm)$ or $(\mcH_i(u),h_{ir})$ for all $r\in \mathbb{N}$. Then, by \eqref{delta_z'}:
\begin{align*}
 (u-v)&\left[\Phi_z(\mcY_i(u)),\Phi_z(\mcX_j^\pm(v))\right]\\
 &=\sum_{n,m\in \mathbb{N}}[y_{i,n},x_{j,m}^\pm](u-v)f_{n,m}(u,v,z)\\
 &=\sum_{n,m\in \mathbb{N}}[y_{i,n},x_{j,m}^\pm](f_{n-1,m}(u,v,z)-f_{n,m-1}(u,v,z)).
\end{align*}
Using \eqref{Y:xh} and \eqref{Y:xx}, we can rewrite the right-hand side as 
\begin{align*}
 \sum_{n,m\in \mathbb{N}}&([y_{i,n+1},x_{j,m}^\pm]-[y_{i,n},x_{j,m+1}^\pm])f_{n,m}(u,v,z)\\
 &=\pm \hbar d_{ij}\sum_{n,m\in \mathbb{N}}\{y_{i,n},x_{j,m}^\pm\}f_{n,m}(u,v,z)
 =\pm \hbar d_{ij}\!\left\{\Phi_z(\mcY_i(u)),\Phi_z(\mcX_j^\pm(v))\right\}\!,
\end{align*}
where $\{x,y\}=xy+yx$. Thus, we have proven that 
\begin{equation*}
 (u-v\mp \hbar d_{ij})\Phi_z(\mcY_i(u))\Phi_z(\mcX_j^\pm(v))=(u-v\pm \hbar d_{ij})\Phi_z(\mcX_j^\pm(v))\Phi_z(\mcY_i(u)),
\end{equation*}
which is precisely \eqref{DY:xh} if $\mcY_i(u)=\mcH_i(u)$, and \eqref{DY:xx} if $\mcY_i(u)=\mcX_i^\pm(u)$.

\medskip  

\noindent \textit{The relation \eqref{DY:xxh}.} Fix $i,j\in \mbI$. Then, by \eqref{Y:xxh}, we have 
\begin{align*}
 [\Phi_z(\mcX_i^+(u)),\Phi_z(\mcX_j^-(v))]=&\delta_{ij}\sum_{n\in \mathbb{N}}h_{i,n}\sum_{k=0}^n \partial_z^{(k)}(u^{-1}\delta(z/u))\partial_z^{(n-k)}(v^{-1}\delta(z/v))\\
 &=\delta_{ij}\sum_{n\in \mathbb{N}}h_{i,n} \partial_z^{(n)}(u^{-1}v^{-1}\delta(z/u)\delta(z/v))\\
 &=\delta_{ij} u^{-1}\delta(v/u)\Phi_z(\mcH_i(v)),
\end{align*}
where in the second equality we have used the generalized Leibniz identity, and in the third equality we have used that 
$u^{-1}v^{-1}\delta(z/u)\delta(z/v)=u^{-1}v^{-1}\delta(v/u)\delta(z/v)$.

\medskip

\noindent \textit{The Serre relations \eqref{DY:Serre}.} Fix $i,j\in \mbI$ with $i\neq j$ and let $m=1-a_{ij}$. For $(n_1,\ldots,n_m,s)\in \mathbb{N}^{m+1}$, set 
\begin{equation*}
 f_{u_1,\ldots,u_m,v}^{n_1,\ldots,n_m,s}(z)=\partial_z^{(n_1)}(u_1^{-1}\delta(z/u_1))\cdots \partial_z^{(n_m)}(u_m^{-1}\delta(z/u_m)) \partial_z^{(s)}(v^{-1}\delta(z/v)).  
\end{equation*}
Then, since $f_{u_1,\ldots,u_m,v}^{n_1,\ldots,n_m,s}(z)$ is symmetric in $\{1,\ldots,m\}$, \eqref{Phi_z} gives 
\begin{align*}
 &\sum_{\pi \in S_{m}} \left[\Phi_z(\mcX_i^{\pm}(u_{\pi(1)})), \left[\Phi_z(\mcX_i^{\pm}(u_{\pi(2)})), \cdots, \left[\Phi_z(\mcX_i^{\pm}(u_{\pi(m)})),\Phi_z(\mcX_j^{\pm}(v))\right] \cdots\right]\right]\\
 &=\sum_{n_1,\ldots,n_m,s\in \mathbb{N}}f_{u_1,\ldots,u_m,v}^{n_1,\ldots,n_m,s}(z)\sum_{\pi\in S_m}\left[x_{i,n_{\pi(1)}}^\pm,\left[x_{i,n_{\pi(2)}}^\pm,\cdots,\left[x_{i,n_{\pi(m)}}^\pm,x_{js}^\pm\right]\cdots\right]\right]=0,
\end{align*}
where the last equality holds by \eqref{Y:Serre}.

\begin{proof}[Proof of \eqref{Phi:3}] Fix $c\in \C^\times$. By Part \eqref{Phi:1} of the theorem and Part \eqref{whYgz:4} of Proposition \ref{P:whYgz}, $\Phi_c=\mathscr{Ev}_c\circ \Phi_z$ is a homomorphism of $\C[\![\hbar]\!]$-algebras satisfying $\Phi_c\circ\iYh=\tau_c$. As we have already established the uniqueness assertion, we are done. 
\end{proof}

\subsection{Consequences and formulas}\label{ssec:Phi-imp}
 
 As a first, and rather immediate, corollary to Theorem \ref{T:Phi} we obtain the injectivity of the natural homomorphisms from $\Yhg$ to both $\DYg$ and $\DYhg$, and deduce the existence of injective translation endomorphisms on the standard Borel subalgebras of $\Yhg$.

\begin{corollary}\label{C:Phi}  Define $\Yhb{\pm}$ to be the subalgebra of $\Yhg$ generated by the Cartan subalgebra $\mfh$ and $\{x_{ir}^\pm, h_{ir}\}_{i\in \mbI, r\in \mathbb{N}}$. Then:
\leavevmode
\begin{enumerate}[font=\upshape]
\item\label{CPhi:1} The algebra homomorphisms 
\begin{equation*}
\iY:\Yhg\to \DYg \; \text{ and }\; \iYh:\Yhg\to \DYhg
\end{equation*}
are injective.
\item\label{CPhi:2}  For each fixed $i\in \mbI$, the assignment 
 \begin{equation*}
  \sigma_i^\pm:\; x_{jr}^\pm\mapsto x_{j,r+\delta_{ij}}^\pm,\quad h_{jr}\mapsto h_{jr},\quad h\mapsto h \quad \forall \; j\in \mbI,\, r\in \mathbb{N}\; \text{ and }\; h\in \mfh
 \end{equation*}
 determines an injective $\C[\hbar]$-algebra endomorphism $\sigma_i^\pm$ of $\Yhb{\pm}$.
\end{enumerate}
\end{corollary}
\begin{proof}
 Since $\Phi_c\circ \iYh=\tau_c$ is injective, $\iYh$ is injective. As $\iYh=\jD\circ \iY$, we can conclude $\iY$ is also injective. This proves Part \eqref{CPhi:1}.

Consider now Part \eqref{CPhi:2}, and fix $c\in \C^\times$ and  $i\in \mbI$. Define 
\begin{equation*}
\sigma_i^\pm:=\tau_{-c}\circ \Phi_c \circ \mbt_i^{\pm 1} \circ \iYh|_{\Yhb{\pm}}.
\end{equation*}
This is an algebra endomorphism of $\Yhb{\pm}$ which agrees with the assignment in the statement of \eqref{CPhi:2}. It is injective since $\mbt_i^{\pm 1}(\iYh(\Yhb{\pm}))\subset \iYh(\Yhb{\pm})$, $\mbt_i$ is an automorphism, and $\Phi_c\circ \iYh$, $\iYh$ and $\tau_{-c}$ are all injective. \qedhere
\end{proof}
\begin{remark}
When $\mfg$ is of finite type or of simply-laced affine type, one can deduce the existence of the algebra endomorphisms $\sigma_i^\pm$ of $\Yhb{\pm}$ by appealing to the fact that $\Yhg$ is known to admit a triangular decomposition, as in \cite{GTL1}*{\S2.6}. Corollary \ref{C:Phi} circumvents the fact that such a decomposition has not yet been established for general $\mfg$. 
\end{remark}
The next result applies the endomorphisms $\sigma_i^\pm$ of $\Yhb{\pm}$ to obtain a useful set of formulas re-expressing the definition of $\Phi_z$ on each generating series $\mcX_i^\pm(u)$.
\begin{corollary}\label{C:Phi-form}
\leavevmode
\begin{enumerate}[font=\upshape]
 \item\label{CPhi:3}  For each $i\in \mbI$, we have 
  \begin{equation*}
  \Phi_z(\mcX_i^\pm(u))=\mathrm{exp}(\sigma_i^\pm\partial_z)( u^{-1}\delta(z/u) x_{i0}^\pm)= u^{-1}\delta\!\left(\frac{z+\sigma_i^\pm}{u}\right)\!(x_{i0}^\pm).
 \end{equation*}
 In particular, for each $k\in \N$ and $\ell\in \Z$, 
\begin{equation*}
 \Phi_z \circ \mbt_i^{\pm\ell} (\mcX_{ik}^\pm)=\exp(\sigma_i^\pm\partial_z)(z^{k+\ell} x_{i0}^\pm)=(z+\sigma_i^\pm)^\ell \tau_z(x_{ik}^\pm),
\end{equation*}
 where $\mbt_i\in \mathrm{Aut}(\DYhg)$ is as in Proposition \ref{P:ti}. 
 %
\item\label{CPhi:4}  For each $i\in \mbI$, we have 
\begin{gather*}
\exp(-z\partial_u)x_i^\pm(u)=\Phi_z(\mcX_i^\pm(u)_+)=\frac{u^{-1}}{1-u^{-1}(z+\sigma_i^\pm)}(x_{i0}^\pm)\\
\exp(-u\partial_z)x_i^\pm(-z)=\Phi_z(\mcX_i^\pm(u)_-)=-\frac{z^{-1}}{1-z^{-1}(u-\sigma_i^\pm)}(x_{i0}^\pm), 
\end{gather*}
where $\mcX_i^\pm(u)_+=\iYh(x_i^\pm(u))$ and 
$
\mcX_i^\pm(u)_-=\mcX_i^\pm(u)_+-\mcX_i^\pm(u). 
$
\end{enumerate}
\end{corollary}
\begin{proof}
Since $(\sigma_i^\pm)^n(x_{i0}^\pm)=x_{in}^\pm$ for all $i\in \mbI$ and $n\in \N$, Part \eqref{CPhi:3} follows directly from \eqref{Phi_z}. As for Part \eqref{CPhi:4}, the leftmost equalities follow from \eqref{Phi_z-tau_z} and \eqref{vertex}, while the rightmost equalites are readily deduced from Part \eqref{CPhi:3}.
\end{proof}
%

\subsection{Similarity of \texorpdfstring{$\Phi_c$}{Phi\_c} and \texorpdfstring{$\Phi_a$}{Phi\_a} }\label{ssec:gr-tw}

We conclude Section \ref{sec:Phi} by establishing that, for any $a,c\in \C^\times$, $\Phi_c$ and $\Phi_a$ are equal up to conjugation by a gradation automorphism governed by the ratio $\frac{c}{a}$.
 
For each $a\in \C^\times$, introduce the $\C$-algebra automorphism $\chi_a$ of $\DYg$ by 
\begin{equation*}
\chi_a=\bigoplus_{k\in \Z} a^k \id_k \in \mathrm{Aut}_\C (\DYg),
\end{equation*}
where $\id_{k}$ is the identity map on the $k$-th graded component $\DYg_k$ of $\DYg$. Since $\chi_a(\hbar^n\DYg)=\hbar^n\DYg$ for each $n\in \N$, $\chi_a$ extends to a $\C$-algebra automorphism
of $\DYhg$, which we again denote by $\chi_a$. These gives rise to an action of the multiplicative group $\C^\times$ on $\DYhg$. That is, one has 
\begin{equation*}
\chi_{a}\circ \chi_{c}=\chi_{a\cdot c} \quad \forall\; a,c\in \C^\times.
\end{equation*}
In addition, $\chi_a$ restricts to a $\C$-algebra automorphism of $\Yhg\cong \iYh(\Yhg)$, which extends by continuity to an automorphism $\chi_{a}^{\iYh}$ of the completed Yangian $\cYhg$. 
The next proposition uses these automorphisms to illustrate the precise relation between $\Phi_c$ and $\Phi_a$, for any $a,c\in \C^\times$. 
\begin{proposition}\label{P:gr-tw}
For each pair of points $a,c\in \C^\times$, one has the identity
\begin{equation*}
\Phi_c=\chi_{a/c}^{\iYh}\circ  \Phi_a \circ \chi_{c/a}.
\end{equation*}
\end{proposition}
\begin{proof}
The composition $\chi_{a/c}^{\iYh}\circ  \Phi_a \circ \chi_{c/a}$ fixes $\hbar$, and is thus a $\C[\![\hbar]\!]$-algebra homomorphism. Next, observe that, for each $b\in \C^\times$, $\chi_b$ satisfies
\begin{equation*}
\chi_b(h)=h,\quad \chi_b(\mcA_i(u))=\mcA_i(u/b) \quad \forall \; h\in \mfh,\,i\in \mbI,
\end{equation*}
where $\mcA_i(u)$ takes value $u\mcX_i^\pm(u)$ or $u\mcH_i(u)$. 
The assertion of the proposition therefore follows from 
\eqref{Phi_z} together with the identity 
\begin{equation*}
(a/c)^n\!\left.\partial_z^{(n)}\!\left( \delta(zc/au)\right)\right|_{z=a}
=\left.\partial_w^{(n)}\!\left( \delta(w/u)\right)\right|_{w=c},
\end{equation*}
which is obtained by making the change of variables $w=zc/a$. \qedhere
\end{proof}
\section{Isomorphism with completed Yangian}\label{sec:Iso}

The evaluation ideal $\mcJ$ at $t=1$ is defined to be the kernel of the composite
\begin{equation*}
 \DYhg\xrightarrow{\hbar\mapsto 0} U(\mft\rtimes \ddot\mfh)\xrightarrow{\evg} U(\mfg),
\end{equation*}
where $\hbar\mapsto 0$ denotes reduction modulo $\hbar$, under the identification of Proposition \ref{P:red-h}, and  $\evg$ is the epimorphism of algebras induced by the composition
\begin{equation}\label{evg}
\mft\rtimes\ddot\mfh\xrightarrow{\pi_t\oplus \id} \dot\mfg[t^{\pm 1}]\rtimes \ddot\mfh \xrightarrow{\evdg\oplus \id} \dot\mfg\rtimes \ddot\mfh\cong \mfg,
\end{equation} 
with $\pi_t$ as in \eqref{pi_a} and $\evdg:\dot\mfg[t^{\pm 1}]\to \dot\mfg$ the evaluation morphism given by $t\mapsto 1$. 
In this section, we will prove that the evaluation of $\Phi_z$ at $z=1$ induces an isomorphism of $\C[\![\hbar]\!]$-algebras
\begin{equation*}
 \widehat{\Phi}: \cDYhg\iso\cYhg,
\end{equation*}
where $\cDYhg$ is the completion of $\DYhg$ with respect to the descending filtration 
\begin{equation*}
 \DYhg=\mcJ^0\supset \mcJ \supset \mcJ^2 \supset \cdots \supset \mcJ^n\supset \cdots
\end{equation*}
%
%
This will be achieved in Theorem \ref{T:hatiso} of Section \ref{ssec:wtPhi}. In Section  \ref{ssec:ch-ev}, we will obtain a generalization of this result which holds for an arbitrary evaluation point $c\in \C^\times$. We will then conclude this section with two applications of Theorem \ref{T:hatiso}: In Section \ref{ssec:ev-Yhg}, we will show that the natural inclusion $\iYh$ extends to an isomorphism between the evaluation completions of $\Yhg$ and $\DYhg$ at $t=1$. We will then demonstrate in Section \ref{ssec:degen} that $\Yhg$ can be realized as a degeneration of $\DYhg$, in the same way that $\Yhg$ can be realized as a degeneration of the quantum loop algebra $U_\hbar(L\mfg)$.
 
\subsection{The isomorphism \texorpdfstring{$\widehat\Phi$}{}}\label{ssec:wtPhi}

In what follows, we shall set $\Phi=\Phi_1$, where $\Phi_1$ is the morphism $\Phi_c$ from Theorem \ref{T:Phi} with $c$ taken to be $1$. Since 
\begin{equation*}
\partial_z^{(n)}(u^{-1}\delta(z/u))=(-1)^n \partial_u^{(n)}(u^{-1}\delta(z/u)) \quad \forall \; n\in \N,
\end{equation*}
we deduce from \eqref{Phi_z} that 
$\Phi$ is given explicitly by the following data: 
\begin{equation}\label{Phi}
\begin{gathered}
\Phi:\DYhg\to \cYhg,\\
 \Phi(h)=h,\quad \Phi(\mcX_i^\pm(u))=\sum_{n\in \mathbb{N}}(-1)^n x_{in}^\pm \partial_u^{(n)}(\delta(u)) \quad \forall \; i\in \mbI \; \text{ and }\; h\in \mfh.
\end{gathered}
\end{equation}

 For each $n\in \N$, introduce the ideal $\cYhgeq{n}\subset \cYhg$ by 
\begin{equation*}
\cYhgeq{n}=\prod_{k\geq n}\Yhg_k.
\end{equation*}
Equivalently, under the identification of Part \eqref{whYgz:1} of Lemma \ref{L:whYhg}, one has 
\begin{equation*}
\cYhgeq{n}=\varprojlim_{k>n}\!\left(\Yhgp^n/\Yhgp^k\right)\!. 
\end{equation*}
\begin{lemma}\label{L:whPhi}
We have 
\begin{equation*}
\Phi(\mcJ^n)\subset \cYhgeq{n} \quad \forall \; n\in \N.
\end{equation*}
Consequently, $\Phi$ induces a homomorphism of $\C[\![\hbar]\!]$-algebras
 \begin{equation*}
  \widehat\Phi:\cDYhg\to \cYhg.
 \end{equation*}
\end{lemma}
\begin{proof}
 We proceed analogously to the proof of \cite{GTL1}*{Thm.~6.2 (1)}. To prove the first assertion, it suffices to show that $\Phi(\mcJ)\subset \cYhgp:=\cYhgeq{1}$, as this will imply 
 \begin{equation*}
 \Phi(\mcJ^n)\subset (\cYhgp)^n\subset \cYhgeq{n} \quad \forall \; n\in \N. 
 \end{equation*}
 As the kernel of the evaluation homomorphism $\evg$ is generated as an ideal by $\{X_{ir}^\pm-X_{is}^\pm\}_{i\in \mbI, r,s\in \Z}$, the ideal $\mcJ\subset \DYhg$ is generated by 
 \begin{equation}\label{J:gen}
  \hbar \DYhg\cup \{\mcX_{ir}^\pm-\mcX_{is}^\pm\}_{i\in \mbI, r,s,\in \Z}.
 \end{equation}
 Since $(u^r-u^s)\delta(u)=0$ for all $r,s\in \Z$, \eqref{Phi} yields
 \begin{equation*}
  (u^r-u^s)\Phi(\mcX_i^\pm(u))=\sum_{n>0}(-1)^n x_{in}^\pm (u^r-u^s)\partial_u^{(n)}(\delta(u))\in \cYhgp[\![u^{\pm 1}]\!]\quad \forall \; i\in \mbI. 
 \end{equation*}
 Applying the formal residue $\mathrm{Res}_u: \cYhg[\![u^{\pm 1}]\!]\to \cYhg$ to this identity, we obtain 
 \begin{equation*}
 \Phi(\mcX_{ir}^\pm)-\Phi(\mcX_{is}^\pm)=\mathrm{Res}_u( (u^r-u^s)\Phi(\mcX_i^\pm(u)))\in \cYhgp. 
 \end{equation*}
 As $\hbar\in \cYhgp$, this completes the proof of the first part of the lemma. 
 
 We may thus conclude that $\Phi$ induces a family of $\C[\![\hbar]\!]$-algebra homomorphisms
 \begin{equation}\label{Phi_n}
  \Phi_{n}:\DYhg/\mcJ^n \to \cYhg/\cYhgeq{n}\cong \Yhg/\Yhgp^n \quad \forall\; n\in \mathbb{N}.
 \end{equation}
 Taking the inverse limit of this family, we obtain 
 \begin{equation*}
 \widehat\Phi=\varprojlim_n \Phi_{n}: \cDYhg\to \cYhg.\qedhere
 \end{equation*}
\end{proof}
%
%
We will show that $\widehat \Phi$ is an isomorphism by constructing its inverse explicitly. Set 
\begin{equation}\label{Ga}
 \Gamma=\iYh\circ \tau_{-1}: \Yhg\into \DYhg,
\end{equation}
where we recall that $\tau_{-1}\in \mathrm{Aut}(\Yhg)$ is defined in \eqref{shift-c}, and $\iYh$ is the natural homomorphism 
$\Yhg\to \DYhg$, which by Corollary \ref{C:Phi} is an embedding. 
\begin{lemma}\label{L:whGa} We have 
\begin{equation*}
\Gamma(\Yhgp)\subset \mcJ.
\end{equation*}
Consequently, $\Gamma$ induces a homomorphism of $\C[\![\hbar]\!]$-algebras 
 \begin{equation*}
  \widehat{\Gamma}: \cYhg\to \cDYhg.
 \end{equation*}
\end{lemma}
\begin{proof}
  On the generating set $\mfh\cup\{x_{i0}^\pm,h_{i1}\}_{i\in \mbI}$, $\Gamma$ is given by 
 \begin{equation*}
  \Gamma(h)=h,\quad  \Gamma(x_{i0}^\pm) =\mcX_{i0}^\pm,\quad
  \Gamma(h_{i1})=\mcH_{i1}-\mcH_{i0}\in \mcJ \quad \forall\; i\in \mbI \; \text{ and }\; h\in \mfh. 
 \end{equation*}
 Since $\mcJ^0=\DYhg$, this implies that $\Gamma(\Yhg_k)\subset \mcJ^k$ for all $k\in \N$, and therefore that 
 $\Gamma(\Yhgp)\subset \mcJ$. Consequently, $\Gamma$ induces a family of $\C[\![\hbar]\!]$-algebra homomorphisms
 \begin{equation}\label{Ga_n}
  \Gamma_{n}:\Yhg/\Yhgp^n\to \DYhg/\mcJ^n \quad \forall \; n\in \N.
 \end{equation}
Taking the inverse limit of this system, we obtain 
\begin{equation*}
 \widehat\Gamma=\varprojlim_n \Gamma_{n}: \cYhg\to \cDYhg. \qedhere
\end{equation*}
\end{proof}
%
%
%
%
In order to prove that $\widehat\Gamma={\widehat\Phi}^{-1}$, we will first analyze how the family of automorphisms $\{\mbt_i\}_{\in \mbI}\subset \mathrm{Aut}(\DYhg)$ of Proposition \ref{P:ti} interact with the ideal $\mcJ\subset \DYhg$. 
\begin{lemma}\label{L:t} Fix $i\in \mbI$. Then
\begin{equation}\label{t-vanish}
\begin{gathered}
\mbt_i(\mcJ)=\mcJ,\\
(\id-\mbt_i^{\pm 1})\mcJ^n \subset \mcJ^{n+1} \quad \forall \; n\in \N.
\end{gathered}
\end{equation}
Consequently, $\mbt_i$ induces $\mbt_{i,n}\in \mathrm{Aut}(\DYhg/\mcJ^n)$ satisfying
\begin{equation}\label{t-1}
\mbt_{i,n}^{\mp 1}=\sum_{k=0}^{n-1}(\id-\mbt_{i,n}^{\pm 1})^k\quad \forall \; n\in \N. 
\end{equation}
\end{lemma}
\begin{proof}
As $\mbt_i$ permutes the generating set \eqref{J:gen} of $\mcJ$, we have $\mbt_i(\mcJ)=\mcJ$. 
It follows that, for each $n\in \N$, $\mbt_i$ induces  $\mbt_{i,n}\in \mathrm{Aut}(\DYhg/\mcJ^n)$ uniquely determined by 
\begin{equation*}
\mathrm{q}_n\circ \mbt_i = \mbt_{i,n}\circ \mathrm{q}_n, 
\end{equation*}
where $\mathrm{q}_n: \DYhg\onto \DYhg/\mcJ^n$ is the natural quotient map. 

If the second relation of \eqref{t-vanish} holds, then $(\id-\mbt_i^{\mp 1})^n\DYhg\subset \mcJ^n$ for each $n\in \N$. In particular, $(\id-\mbt_{i,n}^{\mp 1})^n=0$, from which \eqref{t-1} follow readily. 

We are thus left to prove that $(\id-\mbt_i^{\pm 1})\mcJ^n \subset \mcJ^{n+1}$ for all $n\in \N$. Since 
\begin{equation*}
(\id-\mbt_i^{-1})=(\mbt_i-\id)\mbt_i^{-1} \quad \text{ and }\quad \mbt_i^{-1}(\mcJ)=\mcJ,
\end{equation*}
we need only prove this for $\mbt_i$. Moreover, as $\mbt_i(\mcJ)=\mcJ$ and 
\begin{equation}\label{t_i-der}
(\id-\mbt_i)(x_1x_2\cdots x_n)= \sum_{j=1}^n x_1\cdots x_{j-1} (\id-\mbt_i)(x_j) \mbt_i(x_{j+1})\cdots \mbt_i(x_n)
\end{equation}
for any $x_1,\ldots,x_n\in \DYhg$ and $n>0$, it suffices to prove that
\begin{equation*}
(\id-\mbt_i)\DYhg\subset \mcJ \quad \text{ and }\quad (\id-\mbt_i)\mcJ\subset \mcJ^2.
\end{equation*}
As $\id-\mbt_i$ is $\C[\![\hbar]\!]$-linear and annihilates $1$, \eqref{J:gen} and \eqref{t_i-der} imply that these inclusions will follow from 
\begin{equation*}
(\id-\mbt_i)\mcX_{jr}^\pm \in \mcJ \quad \text{ and }\quad (\id-\mbt_i)(\mcX_{jr}^\pm-\mcX_{js}^\pm)\in \mcJ^2  \quad \forall \;  r,s\in \Z,\; j\in \mbI. 
\end{equation*}
By definition, we have 
\[
(\id-\mbt_i)\mcX_{jr}^\pm=\delta_{ij}(\mcX_{ir}^\pm-\mcX_{i,r\pm 1}^\pm)\in \mcJ.
\] 
In addition, since $\hbar\in \mcJ$ and $\mcH_{i1}-\mcH_{i0}\in \mcJ$, we have 
\begin{align*}
(\id-\mbt_i)(\mcX_{jr}^+-\mcX_{js}^+)&=\delta_{ij}(\mcX_{ir}^+- \mcX_{i,r+1}^+ -\mcX_{is}^+ + \mcX_{i,s+ 1}^+)\\
&=\frac{\delta_{ij}}{2d_i}[\iYh(t_{i1}-h_{i0}),\mcX_{is}^+-\mcX_{ir}^+]\in \mcJ^2,
\end{align*}
where we have used \eqref{Ti-DY} in the second equality. Similarly, 
\begin{equation*}
(\id-\mbt_i)(\mcX_{jr}^--\mcX_{js}^-)=\frac{\delta_{ij}}{2d_i}[\iYh(t_{i1}-h_{i0}),\mcX_{i,s- 1}^--\mcX_{i,r-1}^-]\in \mcJ^2. \qedhere
\end{equation*}
\end{proof} 
\begin{remark}\label{R:tr-ext}
The lemma implies that, for each $i\in \mbI$, $\mbt_i$ extends to an automorphism $\widehat{\mbt}_i$ of the $\C[\![\hbar]\!]$-algebra $\cDYhg$. More precisely, 
\begin{equation*}
\widehat{\mbt}_{i}:=\varprojlim_{n}\mbt_{i,n}\in \mathrm{Aut}(\cDYhg). 
\end{equation*}
In addition, $\widehat{\mbt}_{i}$ satisfies the relation 
\begin{equation*}
\widehat{\mbt}_{i}^{\,\mp 1}=\sum_{k\in \N}(\id-\widehat{\mbt}_i^{\pm 1})^k=\varprojlim_{n} \sum_{k=0}^{n-1}(\id-\mbt_{i,n}^{\pm 1})^k.
\end{equation*}
\end{remark}

With the above lemma at our disposal, we are now prepared to prove the main result of this section. Let $\widehat{\Phi}$ and $\widehat{\Gamma}$ be as in Lemmas \ref{L:whPhi} and \ref{L:whGa}. 
\begin{theorem}\label{T:hatiso}
 The $\C[\![\hbar]\!]$-algebra homomorphisms
 \begin{equation*}
  \widehat\Phi:\cDYhg\to \cYhg \quad \text{ and }\quad \widehat \Gamma:\cYhg\to \cDYhg
 \end{equation*}
 are mutual inverses. In particular, $\widehat\Phi$ is an isomorphism of $\C[\![\hbar]\!]$-algebras. 
  
\end{theorem}
\begin{proof}
Let $\{\Phi_n\}_{n\in \mathbb{N}}$ and $\{\Gamma_n\}_{n\in \mathbb{N}}$ be as in \eqref{Phi_n} and \eqref{Ga_n}, respectively. Since $\widehat \Phi=\varprojlim_n \Phi_n$ and $\widehat \Gamma=\varprojlim_n \Gamma_n$, it suffices to prove that 
\begin{equation}\label{n-inv}
 \Gamma_n=\Phi_n^{-1} \quad \forall\; n\in \N.
\end{equation}
Fix $n\in \N$. As $\Phi\circ \Gamma=\Phi \circ (\iYh \circ \tau_{-1})=\tau_1\circ\tau_{-1}=\id_{\Yhg}$, we have 
\begin{equation*}
 \Phi_{n}\circ \Gamma_{n}=\id\in \End(\Yhg/\Yhgp^n).
\end{equation*}
Hence, \eqref{n-inv} will hold provided that $\Gamma_n$ is surjective, which we prove below.

Since $\tau_{-1}$ is an automorphism we have $\iYh(\Yhg)\subset \mathrm{Im}(\Gamma)$, and thus 
\begin{equation}\label{J-Yhg-gen}
(\mathrm{q}_n\circ\iYh)\Yhg \subset \mathrm{Im}(\Gamma_n), 
\end{equation}
where we recall that $\mathrm{q}_n:\DYhg\onto \DYhg/\mcJ^n$ is the natural quotient map. 
As $\mbt_i^{\pm 1}\iYh(\Yhb{\pm})\subset \iYh(\Yhb{\pm})$, the identity 
\eqref{t-1} of Lemma \ref{L:t} implies that
\begin{equation*}
(\mbt_{i,n}^{\mp 1}\circ \mathrm{q}_n\circ \iYh)\Yhb{\pm}\subset (\mathrm{q}_n\circ \iYh)\Yhb{\pm}.
\end{equation*} 
Therefore, 
\begin{equation*}
 \mathrm{q}_n(\mcX_{i,-k-1}^\pm)=(\mbt_{i,n}^{\mp 1})^{k+1}\mathrm{q}_n(\mcX_{i0}^\pm) \subset (\mathrm{q}_n\circ\iYh)\Yhg \quad \forall\;  k\in \N,\;i\in \mbI.
\end{equation*}
As $\DYhg$ is generated by $\{\mcX_{i-k-1}^\pm\}_{i\in \mbI, k\in \N}\cup \iYh(\Yhg)$, combining the above with \eqref{J-Yhg-gen} yields
\begin{equation*} 
\DYhg/\mcJ^n\subset(\mathrm{q}_n\circ\iYh)\Yhg\subset\mathrm{Im}(\Gamma_n)\subset \DYhg/\mcJ^n. \qedhere
\end{equation*}
\end{proof}
\begin{remark}
Fix $n\in \N$ and $i\in \mbI$. The relation 
\begin{equation*}
(\Gamma_n\circ \Phi_n\circ \mathrm{q}_n)\mcX_{i,-k-1}^\pm=\mathrm{q}_n(\mcX_{i,-k-1}^\pm) \quad \forall\; k\in \N
\end{equation*}
may also be proven directly as follows. Set 
\begin{equation*}
\mcX_i^\pm(u)_{n}^{\Phi}:= -\sum_{k=0}^{n-1} \frac{(-\sigma_i^\pm)^k\partial_u^k}{k!}\left(\frac{x_{i0}^\pm}{1-u}\right)\in \Yhg[\![u]\!].
\end{equation*}
By \eqref{CPhi:4} of Corollary \ref{C:Phi-form}, 
$
\Phi(\mcX_i^\pm(u)_-)\equiv \mcX_i^\pm(u)_{n}^{\Phi}$ modulo $\cYhgeq{n}[\![u]\!]$. 
By \eqref{CPhi:3} of Corollary \ref{C:Phi-form}, 
\begin{equation*}
\tau_c \circ (\sigma_i^\pm)^k (x_{i0}^\pm)=(c+\sigma_i^\pm)^kx_{i0}^\pm \quad \forall \; k\in \N,\; c\in \C^\times. 
\end{equation*} 
It follows that 
\begin{equation*}
\Gamma(\mcX_i^\pm(u)_{n}^{\Phi})=\sum_{k=0}^{n-1} \frac{(\id-\mbt_i^{\pm 1})^k\partial_u^k}{k!}\!\left(\frac{\mcX_{i0}^\pm}{u-1}\right)
=-\sum_{p\geq 0} u^p \sum_{k=0}^{n-1} \binom{p+k}{p}  (\id-\mbt_i^{\pm 1})^k \mcX_{i0}^\pm.
\end{equation*}
Applying $\mathrm{q}_n$ and using \eqref{t-1} together with $(\id-\mbt_{i,n}^{\pm 1})^n=0$, we obtain
\begin{equation*}
(\Gamma_n\circ \Phi_n\circ \mathrm{q}_n)\mcX_i^\pm(u)_-=
\mathrm{q}_n\circ \Gamma (\mcX_i^\pm(u)_{n}^{\Phi})=-\sum_{p\geq 0} \mbt_{i,n}^{\mp (p+1)}(\mcX_{i0}^\pm)u^p=\mathrm{q}_n(\mcX_i^\pm(u)_-). 
\end{equation*}
\end{remark}
%
%
\subsection{Change of evaluation point}\label{ssec:ch-ev}
We now apply Proposition \ref{P:gr-tw} to illustrate that the evaluation point $c=1=t$ can be replaced by any $c\in \C^\times$ in the statement of Theorem \ref{T:hatiso}. 
Let 
$\mcJ_c\subset \DYhg$ be the evaluation ideal at $t=c$. That is, $\mcJ_c$ is the kernel of the composite 
\begin{equation*}
 \DYhg\xrightarrow{\hbar\mapsto 0} U(\mft\rtimes \ddot\mfh)\xrightarrow{\evg^c} U(\mfg),
\end{equation*}
where $\evg^c$ is defined as in \eqref{evg}, but with $\evdg$ replaced by the evaluation morphism $\evdg^c:\dot\mfg[t^{\pm 1}]\to \dot\mfg$ given by $t\mapsto c$. Note that the above composition coincides with 
\begin{equation*}
 \DYhg\xrightarrow{\chi_c}\DYhg\xrightarrow{\hbar\mapsto 0} U(\mft\rtimes \ddot\mfh)\xrightarrow{\evg} U(\mfg),
\end{equation*}
where $\chi_c$ is as in Proposition \ref{P:gr-tw}. In particular, we have the equality $\chi_c(\mcJ_c)=\mcJ$ and may thus that $\chi_c$ induces an isomorphism of $\C$-algebras 
\begin{equation*}
\wh{\chi}_c:\cDYhg_c\iso \cDYhg,
\end{equation*}
where $\cDYhg_c$ is the completion of $\DYhg$ with respect to its descending filtration given by powers of $\mcJ_c$. The following corollary then provides the desired generalization of Theorem \ref{T:hatiso}. 
\begin{corollary}\label{C:hatiso}
Fix $c\in \C^\times$. Then:
\begin{enumerate}[font=\upshape]
\item\label{Chatiso:1} $\Phi_c$ satisfies 
$
\Phi_c(\mcJ_c)\subset \cYhg_+
$
and thus induces a $\C[\![\hbar]\!]$-algebra homomorphism
\begin{equation*}
  \widehat{\Phi}_c:\cDYhg_c\to \cYhg.
\end{equation*}
\item\label{Chatiso:2} $\widehat{\Phi}_c$ is an isomorphism satisfying the relations
\begin{equation*}
\widehat{\Phi}_c=\chi_{1/c}^{\iYh}\circ  \wh{\Phi} \circ \wh{\chi}_c \quad 
\text{ and }
\quad 
\widehat{\Phi}_c^{-1}=\wh{\chi}_{1/c}\circ \wh{\Gamma} \circ \chi_{c}^{\iYh}. 
\end{equation*}
\end{enumerate}
\end{corollary}
\begin{proof}
Since $\chi_c(\mcJ_c)=\mcJ$, Proposition \ref{P:gr-tw} and Lemma \ref{L:whPhi} yield 
\begin{equation*}
\Phi_c(\mcJ_c)=(\chi_{1/c}^{\iYh}\circ \Phi)(\mcJ)\subset \cYhg_+.
\end{equation*}
As in the proof of Lemma \ref{L:whPhi}, this implies that $\Phi_c$ gives rise to $\wh{\Phi}_c$ as in the statement of Part \eqref{Chatiso:1}. As for Part \eqref{Chatiso:2}, it follows by continuity that $\wh{\Phi}_c$ and $\chi_{1/c}^{\iYh}\circ  \wh{\Phi} \circ \wh{\chi}_c$ are both determined by their values on the image of $\DYhg$, and therefore coincide by Proposition \ref{P:gr-tw}. The assertion that $\wh{\Phi}_c$ is invertible with inverse $\wh{\chi}_{1/c}\circ \wh{\Gamma} \circ \chi_{c}^{\iYh}$ now follows immediately from Theorem \ref{T:hatiso}. 
\end{proof}
%
%
\subsection{The evaluation completion of \texorpdfstring{$\Yhg$}{Yg}}\label{ssec:ev-Yhg}

A simplification of the classical result underlying Theorem \ref{T:hatiso} is provided by the observation that the translation $w\mapsto t+1$ induces an isomorphism 
\begin{equation*}
\wh{\C[w^{\pm 1}]}\iso \C[\![t]\!], \quad \text{ where }\quad  \wh{\C[w^{\pm 1}]}=\varprojlim_n(\C[w^{\pm 1}]/\mathbb{J}^n)
\end{equation*}
and $\mathbb{J}=(w-1)$ is the evaluation ideal of $\C[w^{\pm 1}]$ at $w=1$. On the other hand, the natural inclusion $\C[t]\into \C[w^{\pm 1}]$ sending $t$ to $w$ extends to an isomorphism
\begin{equation*}
\wh{\C[t]}\iso \wh{\C[w^{\pm 1}]},
\end{equation*}
where $\C[t]$ is completed with respect to its evaluation ideal $\mathbb{J}_{\scriptscriptstyle{+}}=(t-1)$. In this subsection, we prove the Yangian analogue of this fact.
To begin, we define the evaluation ideal $\mcJY\subset \Yhg$ at $t=1$ to be the kernel of the composite 
\begin{equation*}
 \Yhg\xrightarrow{\hbar\mapsto 0} U(\mfs\rtimes \ddot\mfh)\to U(\mfg),
\end{equation*}
where the second arrow is the epimorphism induced by the composition
\begin{equation*}
\mfs\rtimes\ddot\mfh\xrightarrow{\pi_s\oplus \id} \dot\mfg[t]\rtimes \ddot\mfh \xrightarrow{\evdg\oplus \id} \dot\mfg\rtimes \ddot\mfh\cong \mfg,
\end{equation*} 
with $\pi_s$ as in \eqref{pi_a} and $\evdg$ as defined at the beginning of Section \ref{sec:Iso}. Since the above coincides with the composite 
\begin{equation*}
\Yhg\xrightarrow{\iYh} \DYhg\xrightarrow{\hbar\mapsto 0} U(\mft\rtimes \ddot\mfh)\xrightarrow{\evg} U(\mfg),
\end{equation*}
one has the equality $\iYh(\mcJY)=\mcJ\cap \iYh(\Yhg) \subset \mcJ$. Consequently, $\iYh$ induces a  homomorphism of $\C[\![\hbar]\!]$-algebras 
\begin{equation*}
\wh{\iYh}:\eYhg\to \cDYhg,
\end{equation*}
where $\eYhg$ is the completion of $\Yhg$ with respect to the descending filtration
\begin{equation*}
\Yhg=\mcJY^0\supset \mcJY\supset \cdots \supset \mcJY^n\supset \cdots 
\end{equation*}
Our goal is to prove that $\wh{\iYh}$ is an isomorphism using our work in the previous section. We begin with the following analogue of Lemma \ref{L:whGa}, which asserts that the gradation and evaluation completions of $\Yhg$ are isomorphic. 
\begin{lemma} Setting $\mathsf{\Gamma}=\tau_{-1}$, we have the equality
\begin{equation*}
\mathsf{\Gamma}(\Yhgp)= \mcJY.
\end{equation*}
Consequently, $\mathsf{\Gamma}$ induces an isomorphism of $\C[\![\hbar]\!]$-algebras 
 \begin{equation*}
  \widehat{\mathsf{\Gamma}}: \cYhg\iso \eYhg.
 \end{equation*}
\end{lemma}
\begin{proof}
The proof that $\Gamma=\imath\circ \mathsf{\Gamma}$ satisfies $\Gamma(\Yhgp)\subset \mcJ$, given in Lemma \ref{L:whGa}, shows that $ \mathsf{\Gamma}(\Yhgp)\subset \mcJY$. Conversely, by Lemma \ref{L:whPhi}, we have $\Phi(\mcJ)\subset \cYhgp$, and thus 
\begin{equation*}
\mathsf{\Gamma}^{-1}(\mcJY)=\tau_1(\mcJY)=\Phi(\iYh(\mcJY))\subset \Yhg\cap \cYhgp=\Yhgp. \qedhere
\end{equation*}
\end{proof}
Combining this lemma with Theorem \ref{T:hatiso} yields an isomorphism of $\C[\![\hbar]\!]$-algebras 
\begin{equation*}
\widehat{\mathsf{\Gamma}}\circ\widehat{\Phi}: \cDYhg\iso \eYhg.
\end{equation*}
It follows from the identity $\mathsf{\Gamma}\circ \Phi\circ \iYh =\tau_{-1}\circ \tau_1=\id_{\Yhg}$ that $\wh{\iYh}$ is a right inverse, and therefore the unique inverse, of this isomorphism. We have thus proven the following corollary, which realizes our current goal. 
\begin{corollary}\label{C:YJ->DJ}
The $\C[\![\hbar]\!]$-algebra homomorphisms
\begin{equation*}
\wh{\iYh}:\eYhg\to \cDYhg \quad \text{ and }\quad \widehat{\mathsf{\Gamma}}\circ\widehat{\Phi}: \cDYhg\iso \eYhg
\end{equation*}
are mutual inverses. In particular, $\wh{\iYh}$ is an isomorphism of $\C[\![\hbar]\!]$-algebras. 
\end{corollary} 
This corollary affords the evaluation completion of $\DYhg$ a rather explicit description. Namely, it coincides with the gradation completion of $\Yhg$ with respect to the shifted $\N$-grading 
\begin{equation*}
\Yhg=\bigoplus_{n\in \N}\mathsf{\Gamma}(\Yhg_{n}),
\end{equation*}
where $\mathsf{\Gamma}=\tau_{-1}$, as above. More precisely, one has the equality 
\begin{equation*}
\cDYhg=\prod_{n\in \N}(\iYh \circ \mathsf{\Gamma})(\Yhg_{n})=\prod_{n\in \N}\Gamma(\Yhg_{n}).
\end{equation*}
The natural homomorphism $\DYhg\to \cDYhg$ can be expressed in terms of these coordinates using Corollary \ref{C:YJ->DJ}. Alternatively, using Remark \ref{R:tr-ext} we find that, for each $i\in \mbI$, one has the identity
\begin{equation}\label{eq-deltas}
u^{-1}\delta\!\left(\widehat{\mbt}_i^{\pm 1}/u\right)=\left.\delta(u+z)\right|_{z=1-\widehat{\mbt}_i^{\pm 1}} 
\end{equation}
in $\End(\cDYhg)[\![u^{\pm 1}]\!]$.
The right-hand side may be rewritten as 
\begin{equation*}
\left.\delta(u+z)\right|_{z=1-\widehat{\mbt}_i^{\pm 1}}= \exp((1-\widehat{\mbt}_i^{\pm 1})\partial_u)\delta(u) =\sum_{n\in \N}  (1-\widehat{\mbt}_i^{\pm 1})^n \partial_u^{(n)}(\delta(u)).
\end{equation*}
Applying both sides of \eqref{eq-deltas} to $\mcX_{i0}^\pm=\iYh(x_{i0}^\pm)$ therefore yields 
\begin{align*}
\mcX_i^\pm(u)=u^{-1}\delta\!\left(\widehat{\mbt}_i^{\pm 1}/u\right)\!\mcX_{i0}^\pm&=\sum_{n\in \N}  (1-\widehat{\mbt}_i^{\pm 1})^n(\mcX_{i0}^\pm) \partial_u^{(n)}(\delta(u))\\
&=\sum_{n\in \N}  (-1)^n\Gamma(x_{in}^\pm) \partial_u^{(n)}(\delta(u))
\end{align*}
in $\cDYhg$, where in the last equality we have applied the second identity of Part \eqref{CPhi:3} of Corollary \ref{C:Phi-form} with $z=-1$, $\ell=n$ and $k=0$. Note that, by \eqref{Phi}, the above computation recovers the identity $\mcX_i^\pm(u)=(\wh{\iYh}\circ \wh{\mathsf{\Gamma}}\circ \wh{\Phi})(\mcX_i^\pm(u))$ of Corollary \ref{C:YJ->DJ}.

\subsection{Degeneration}\label{ssec:degen}

It was observed by Drinfeld \cite{DrQG}*{\S6} and later proven by Guay and Ma \cite{GM} that the Yangian of a finite-dimensional simple Lie algebra $\mfg$ may be viewed as a degeneration of the corresponding quantum loop algebra $U_\hbar(L\mfg)$. In the form presented in \cite{GM} and \cite{GTL1} this result can be stated as follows. The quantum loop algebra $U_\hbar(L\mfg)$ admits a descending filtration given by powers of its evaluation ideal $\mathbf{J}$ at $t=1$, and there is an isomorphism of graded $\C[\hbar]$-algebras 
\begin{equation*}
\mathrm{gr}_{\mathbf{J}}(U_\hbar(L\mfg))\iso \Yhg.
\end{equation*}
It was shown in \cite{GTL1}*{Prop.~6.5} that this isomorphism can be realized 
as the associated graded map $\mathrm{gr}(\Phi_{\scriptscriptstyle{\mathsf{GTL}}})$ of the algebra homomorphism 
\begin{equation*}
 \Phi_{\scriptscriptstyle{\mathsf{GTL}}}:U_\hbar(L\mfg)\to \cYhg
\end{equation*}
of geometric type constructed in \cite{GTL1}. In this section, we present a $\DYhg$-analogue of this result  in which $U_\hbar(L\mfg)$ and $\Phi_{\scriptscriptstyle{\mathsf{GTL}}}$ are replaced by $\DYhg$ and $\Phi$, and $\mfg$ is taken to be an arbitrary symmetrizable Kac--Moody algebra.

In what follows, we view  $\DYhg$, $\Yhg$ and $\cYhg$ as $\N$-filtered algebras, with descending filtrations 
\begin{equation*}
\{\mcJ^n\}_{n\in \N}, \quad \{\Yhgp^n\}_{n\in \N}\quad \text{ and }\quad \{\cYhgeq{n}\}_{n\in \N},
\end{equation*}
respectively. Note that
\begin{equation*}
 \mathrm{gr}(\cYhg)\cong \mathrm{gr}\Yhg=\bigoplus_{n\in \N}\left.\Yhgp^n\right/{\Yhgp^{n+1}}\cong \bigoplus_{n\in \N} \Yhg_n=\Yhg, 
\end{equation*}
as graded $\C[\hbar]$-algebras. Let $\Phi$ and $\Gamma$ be as in \eqref{Phi} and \eqref{Ga}, respectively. 
\begin{corollary}\label{C:degen}
 $\Gamma:\Yhg\to \DYhg$ is filtered and the induced homomorphism
 \begin{equation*}
  \mathrm{gr}(\Gamma): \Yhg\to \mathrm{gr}(\DYhg).
 \end{equation*}
 is an isomorphism of graded $\C[\hbar]$-algebras with inverse given by 
 \begin{equation*}
  \mathrm{gr}(\Phi):\mathrm{gr}(\DYhg)\to \mathrm{gr}(\cYhg)\cong \Yhg.
 \end{equation*}

\end{corollary}
\begin{proof}
 By Lemma \ref{L:whGa}, $\Gamma(\Yhgp)\subset \mcJ$, and hence $\Gamma$ is filtered. Similarly, Lemma \ref{L:whPhi} implies that 
 $\Phi$ is a filtered morphism. 
 
 In the proof of Theorem \ref{T:hatiso} we showed that
 \begin{equation*}
  \Gamma_{n}:\Yhg/\Yhgp^n\to \DYhg/\mcJ^n \quad \text{ and }\quad
  \Phi_n:\DYhg/\mcJ^n\to \Yhg/\Yhgp^n
 \end{equation*}
 are mutual inverses for each $n\in \mathbb{N}$. By Lemmas \ref{L:whPhi} and \ref{L:whGa}, we have 
 \begin{equation*}
  \Gamma_{n+1}(\Yhgp^{n}/\Yhgp^{n+1})=\mcJ^{n}/\mcJ^{n+1} \quad \forall\; n\in \N.
 \end{equation*}
 Letting $\Gamma_{(n+1)}$ and $\Phi_{(n+1)}$ denote the restrictions of $\Gamma_n$ and $\Phi_n$ to $\Yhgp^{n}/\Yhgp^{n+1}$ and $\mcJ^n/\mcJ^{n+1}$, respectively, we find that
 \begin{equation*}
  \mathrm{gr}(\Gamma)=\oplus_{n\in \N} \Gamma_{(n+1)}: \Yhg\cong\bigoplus_{n\in \N}\Yhgp^{n}/\Yhgp^{n+1}\to \mathrm{gr}(\DYhg)=
  \bigoplus_{n\in \N}\mcJ^{n}/\mcJ^{n+1}
 \end{equation*}
 is an isomorphism with inverse $\mathrm{gr}(\Phi)=\oplus_{n\in \mathbb{N}} \Phi_{(n+1)}$. \qedhere
\end{proof}

The above result can be rephrased in the language of one-parameter deformations as follows. Let $Y_{\hbar,v}(\mfg)$ be the $\C[\hbar,v]$-subalgebra of $\DYhg[v^{\pm 1}]$ generated by $v^{-1}\mcJ$ and $\DYhg$. Equivalently, $Y_{\hbar,v}(\mfg)$ is the extended Rees algebra 
\begin{equation*}
Y_{\hbar,v}(\mfg)=\bigoplus_{n\in \Z}v^{-n}\mcJ^n \subset \DYhg[v^{\pm 1}],
\end{equation*}
where $\mcJ^{-n}=\DYhg$ for all $n\in \N$. Then, by Corollary \ref{C:degen}, $Y_{\hbar,v}(\mfg)$ is a flat deformation of the Yangian $\Yhg$ over $\C[v]$. Indeed, $Y_{\hbar,v}(\mfg)\subset \DYhg[v^{\pm 1}]$ is a  torsion free $\C[v]$-module and one has 
\begin{equation*}
Y_{\hbar,v}(\mfg)/v Y_{\hbar,v}(\mfg)\cong \bigoplus_{n\in \Z}\mcJ^n/\mcJ^{n+1}=\mathrm{gr}(\DYhg)\underset{\mathrm{gr}(\Phi)}{\cong }\Yhg.
\end{equation*}
%

\section{\texorpdfstring{$\DYhg$}{DYg} as a flat deformation}\label{sec:PBW}
We now apply our construction to prove a Poincar\'{e}--Birkhoff--Witt theorem for $\DYhg$, applicable when $\mfg$ is of finite type or of simply-laced affine type. The precise statement of this result, given in Theorem \ref{T:PBW}, simultaneously establishes the injectivity of both $\Phi_z$ and $\Phi_c$  for all such $\mfg$, and therefore that $\DYhg$ can be viewed as both a subalgebra of $\LzhYhg\subset \cYhg[\![z^{\pm 1}]\!]$ and of $\cYhg$. Here we note that the injectivity of $\Phi$ is not an immediate consequence of Theorem \ref{T:hatiso}, since the natural homomorphism 
\begin{equation*}
\DYhg\to \cDYhg
\end{equation*}
has kernel equal to the intersection of all powers $\mcJ^n$ of $\mcJ$, which need not vanish. Theorem \ref{T:PBW} nevertheless implies that this intersection is indeed trivial, at least in the finite and simply-laced affine cases. 

\subsection{The classical limit of \texorpdfstring{$\Phi$}{Phi}}\label{ssec:cl-Phi}
Since the quotient map $\Yhg\to \Yhg/\hbar\Yhg\cong U(\mfs\rtimes \ddot\mfh)$ is $\N$-graded, it induces an isomorphism
\begin{equation*}
\cYhg/\hbar\cYhg\iso \cUsh,
\end{equation*}
where $\mfs_\mfh:=\mfs\rtimes\ddot\mfh$ and $\cUsh$
is the formal completion of $U(\mfs_\mfh)$ with respect to its $\N$-grading: 
\begin{equation*}
\cUsh=\prod_{n\in \N}U(\mfs_\mfh)_n.
\end{equation*}
The classical limit $\bar\Phi$ of $\Phi$ is then the homomorphism $U(\mft\rtimes \ddot\mfh)\to\cUsh$ uniquely determined by the requirement that the following diagram commute:
\tikzcdset{every label/.append style = {font = \small}}
\begin{equation}\label{dia:barPhi}
\begin{tikzcd}[column sep=15ex]
 \DYhg \arrow[two heads]{d} \arrow{r}{\Phi} &  \cYhg\arrow[two heads]{d}\\
 U(\mft\rtimes \ddot\mfh) \arrow{r}{\bar\Phi}       & \cUsh
\end{tikzcd}
\end{equation}
where the vertical arrows are given by reducing modulo $\hbar$. By \eqref{Phi}, $\bar\Phi$ is given explicitly by the formulas 
\begin{gather*}
 \bar\Phi(h)=h,\quad \bar\Phi(X_i^\pm(u))=\sum_{n\in \mathbb{N}}(-1)^n X_{in}^\pm \partial_u^{(n)}(\delta(u)) \quad \forall \; i\in \mbI \; \text{ and }\; h\in \mfh,\\
 \text{ where }\; X_i^\pm(u)=\sum_{r\in \Z}X_{ir}^\pm u^{-r-1}\in \mft[\![u^{\pm 1}]\!].
\end{gather*}
\begin{remark}
Since $\Phi\circ \iYh=\tau_1$ admits an invertible classical limit, it follows that the classical limit of $\iYh$, which coincides with the natural homomorphism 
\begin{equation*}
U(\mfs\rtimes\ddot\mfh)\to U(\mft\rtimes\ddot\mfh),
\end{equation*}
is injective. This justifies our use of the same notation for generators of $\mfs$ and $\mft$.
\end{remark}
The above formulas for $\bar\Phi$ imply that $\bar\Phi(\mft)\subset \wh{\mfs}$, where $\wh{\mfs}$ is the Lie algebra  
\begin{equation*}
\wh{\mfs}=\prod_{n\in \N}\mfs_{n}\subset \cUs. 
\end{equation*}
We may therefore define $\phi$ to be the homomorphism of Lie algebras 
\begin{equation*}
\phi:=\bar\Phi|_\mft:\mft \to \wh{\mfs}.
\end{equation*}

Consider now the injection $\C[w]\into \C[\![t]\!]$ given by $w\mapsto t+1$. As $t+1$ is invertible in $\C[\![t]\!]$ with inverse
\begin{equation*}
 (t+1)^{-1}=\sum_{k\geq 0}(-1)^k t^k,
\end{equation*}
this morphism uniquely extends to $\gamma:\C[w^{\pm 1}]\into \C[\![t]\!]$. We thus obtain an injective homomorphism of Lie algebras
\begin{equation*}
 \id\otimes \gamma:\dot\mfg[w^{\pm 1}]=\dot\mfg\otimes \C[w^{\pm 1}]\into \dot\mfg[\![t]\!]=\dot\mfg\otimes \C[\![t]\!].
\end{equation*}
The homomorphism $\phi$ then satisfies the  commutative diagram
\tikzcdset{every label/.append style = {font = \small}}
\begin{equation}\label{ext-limit}
\begin{tikzcd}[column sep=15ex]
 \mft \arrow[two heads]{d}{\pi_\mft} \arrow{r}{\phi} &  \wh{\mfs}\arrow[two heads]{d}{\widehat{\pi}_\mfs}\\
  \dot{\mfg}[w^{\pm 1}] \arrow{r}{\id\otimes \gamma}       & \dot\mfg[\![t]\!]
\end{tikzcd}
\end{equation}
where $\pi_\mfs$ and $\pi_\mft$ are the graded homomorphisms defined in \eqref{pi_a}, and $\widehat{\pi}_\mfs$ is obtained from $\pi_\mfs$ by extending by continuity. 
We end this subsection by noting that, when $\mfg$ is of finite type, the vertical arrows in \eqref{ext-limit} are isomorphisms and, consequently, $\phi$ is injective. We will see below that this holds in a much more general context.

\subsection{\texorpdfstring{$\DYhg$}{DYg} as a flat deformation}
The following theorem is the main result of this section. 
\begin{theorem}\label{T:PBW}
Suppose that $\phi: \mft\to \widehat{\mfs}$ is injective and that $\Yhdg$ is a torsion free $\C[\hbar]$-module. Then: 
\begin{enumerate}[font=\upshape]
\item\label{PBW:1} For any fixed $c\in \C^\times$, the algebra homomorphisms 
\begin{equation*}
\Phi_z:\DYhg\to \LzhYhg\quad \text{ and }\quad \Phi_c:\DYhg\to \cYhg
\end{equation*} 
are injective. 
\item\label{PBW:2} $\DYhdg$ and $\DYhg$ are flat deformations of $U(\mft)$ and $U(\mft\rtimes \ddot\mfh)$, respectively, over $\C[\![\hbar]\!]$. In particular, there are isomorphisms of $\C[\![\hbar]\!]$-modules 
\begin{equation*}
\DYhdg\cong U(\mft)[\![\hbar]\!]\quad \text{ and }\quad \DYhg\cong U(\mft\rtimes\ddot\mfh)[\![\hbar]\!].
\end{equation*}
\end{enumerate}
Moreover, the hypotheses on $\phi$ and $\Yhdg$ are satisfied whenever $\mfg$ is of finite type or simply-laced affine type. 
\end{theorem}
\begin{remark}
In fact, we will show that $\phi$ is injective whenever $\mfg$ is of untwisted affine type with underlying finite-dimensional simple Lie algebra $\bar\mfg\ncong \mfsl_2$.  Although $\Yhdg$ is expected to be torsion free for all such $\mfg$, this remains a conjecture. 
\end{remark}
\begin{proof}
Let us first clarify why the hypotheses on $\phi$ and $\Yhdg$ are satisfied in the claimed cases. 
That $\Yhdg$ is torsion free when $\mfg$ is of finite type is due to Levendorskii \cite{LevPBW}
(see also \cite{FiTs19}*{Thm.~B.6} and \cite{GRWEquiv}*{Prop.~2.2}).
We have seen that $\phi$ is injective in this case at the end of Section \ref{ssec:cl-Phi}.  

It has recently been proven independently in \cite{GRWvrep} and \cite{YaGuPBW} that $\Yhdg$ is a torsion free $\C[\hbar]$-module when $\mfg$ is of simply-laced affine type. 
We will prove that $\phi$ is injective when $\mfg$ is of untwisted affine type 
in Section \ref{ssec:aff-III} using the identifications $\mfs\cong \mathfrak{uce}(\dot\mfg[t])$ and $\mft_{\kappa}\cong \mathfrak{uce}(\dot\mfg[t^{\pm 1}])$, which are made concrete in Sections \ref{ssec:aff-I}--\ref{ssec:aff-II}. 

Now let us turn to proving \eqref{PBW:1} and \eqref{PBW:2} under the assumption that $\mfg$ is such that 
the hypotheses on $\Yhdg$ and $\phi$ hold. 

\begin{proof}[Proof of \eqref{PBW:1}]\let\qed\relax
By  Part \eqref{Phi:3} of Theorem \ref{T:Phi}  and Proposition \ref{P:gr-tw}, we have 
\begin{equation*}
\mathscr{Ev}_1\circ \Phi_z=\Phi \quad \text{ and }\quad \Phi_c=\chi_{1/c}^\iYh \circ \Phi \circ \chi_c \quad \forall\; c\in \C^\times. 
\end{equation*}
Therefore, it suffices to show that $\Phi$ is injective.
 Taking the direct sum of  $\phi$ with the identity map $\id$ on $\ddot\mfh$, we obtain an injective homomorphism of Lie algebras 
\begin{equation*}
\phi\oplus \id: \mft\rtimes \ddot\mfh\to \wh{\mfs}\rtimes \ddot\mfh,
\end{equation*}
where the action of $\ddot\mfh$ on $\wh{\mfs}$ is obtained from that of $\ddot\mfh$ on $\mfs$ by extending by continuity. By the Poincar\'{e}--Birkhoff--Witt theorem for enveloping algebras, the above map induces an injective homomorphism of algebras 
\begin{equation*}
U(\mft\rtimes \ddot\mfh)\to U(\wh{\mfs}\rtimes \ddot\mfh)\subset \cUsh,
\end{equation*} 
which is precisely the classical limit $\bar\Phi$ of $\Phi$ introduced in Section \ref{ssec:cl-Phi}. In particular, $\bar\Phi$ is injective. 

Suppose now that $x\in \DYhg$ is nonzero. We will employ a standard argument to show that $x\notin\mathrm{Ker}(\Phi)$. Since $\DYhg$ is a separated $\C[\![\hbar]\!]$-module, there is $k\in \N$ such that 
\begin{equation*}
x=\hbar^ky, \quad \text{ where }\;y\notin \hbar\DYhg.
\end{equation*}
Since the image of $y$ in $\DYhg/\hbar\DYhg$ is nonzero and $\bar\Phi$ is injective, the commutativity of the diagram \eqref{dia:barPhi} implies that $y\notin\mathrm{Ker}(\Phi)$. 
Moreover,
as $\Yhg=\Yhdg\rtimes U(\ddot\mfh)$ is torsion free, Part \eqref{whYhg:3} of Lemma \ref{L:whYhg} implies that $\cYhg$ is torsion free. We may thus conclude that 
\begin{equation*}
\Phi(x)=\hbar^k\Phi(y)\neq 0,
\end{equation*}
and therefore that $\Phi$ is injective, as desired. \qedhere
\end{proof}

\begin{proof}[Proof of \eqref{PBW:2}]
By Proposition \ref{P:red-h}, $\DYhdg$ and $\DYhg$ are deformations of $U(\mft)$ and $U(\mft\rtimes \ddot\mfh)$, respectively, over $\C[\![\hbar]\!]$.

 To prove that they are flat deformations, it suffices to show that they are separated, complete and torsion free  $\C[\![\hbar]\!]$-modules (see \cite{KasBook95}*{Prop.~XVI.2.4}, for instance). They are separated and complete by definition, having been defined topologically in terms of generators and relations. They are torsion free since $\Phi$ is injective and 
$\cYhg$ is torsion free, as explained in the proof of \eqref{PBW:1}. \qedhere
\end{proof}

\let\qed\relax

\end{proof}

\subsection{Kassel's realization}\label{ssec:aff-I}

For the remainder of Section \ref{sec:PBW}, we assume that $\mfg$ is an untwisted affine Kac--Moody algebra with underlying simple Lie algebra $\bar\mfg\ncong\mfsl_2$. We then have 
\begin{equation*}
\dot\mfg\cong \gfin[v^{\pm 1}]\oplus \C c
\end{equation*}
as a vector space, with Lie bracket determined by $[c,\dot\mfg]=0$ and 
\begin{equation*}
[x\otimes v^r,y\otimes v^s]=[x,y]\otimes v^{r+s}+r\delta_{r,-s}(x,y)c
\end{equation*}
for all $x,y\in \gfin$ and $r,s\in \Z$. 
By \cite{GRWvrep}*{Prop.~4.7}, there are isomorphisms 
\begin{equation}\label{GRW-4.7}
\uce(\dot\mfg[t])\cong \uce(\gfin[v^{\pm 1},t]) \quad \text{ and }\quad
\uce(\dot\mfg[t^{\pm 1}])\cong \uce(\gfin[v^{\pm 1},t^{\pm 1}]). 
\end{equation}

To prove that $\phi$ is injective, we shall make use of isomorphisms 
\begin{equation*}
\mft_\kappa\iso \uce(\gfin[v^{\pm 1},t^{\pm 1}])
\quad \text{ and }\quad \mfs\iso \uce(\gfin[v^{\pm 1},t])
\end{equation*}
 obtained in the work of Moody--Rao--Yokonuma \cite{MRY90}, which coupled with \eqref{GRW-4.7} yield \eqref{st-uce}. Their construction is partly based on a general result due to Kassel \cite{Kas84}, which provides an explicit realization of $\uce(\gfin\otimes A)$, where $A$ is an arbitrary commutative, associative algebra over the complex numbers. For the sake of completeness, we recall some of the relevant general theory below, beginning with Kassel's realization.
 After briefly discussing some auxiliary properties of this realization in Section \ref{ssec:aff-I.5}, we will review the relevant results from \cite{MRY90}  in Section \ref{ssec:aff-II}.

Let $(\Omega(A),d)$ be the module of K\"{a}hler differentials associated to $A$. That is, $\Omega(A)$ is the $A$-module
\begin{equation*}
\Omega(A)=(A\otimes A)/M,
\end{equation*}
where $A$ acts on $A\otimes A$ by left multiplication in the first tensor factor, and $M$ is the submodule of $A\otimes A$ generated by $1\otimes ab-a\otimes b -b\otimes a$ for all $a,b\in A$. The differential map $d$ is then the derivation
\begin{equation*}
d:A\to \Omega(A), \quad d(a)=1\otimes a \mod  M \quad \forall \; a\in A. 
\end{equation*}
In particular, we can (and will) write $ad(b)$ for the equivalence class of the tensor $a\otimes b$ in $\Omega(A)$. We shall use the same notation for generators of the quotient 
\[
\mfz(A):=\Omega(A)/d(A).
\] 
Consider the alternating bilinear map  $\veps:(\gfin\otimes A)\times (\gfin\otimes A)\to \mfz(A)$ determined by 
\[
\veps(x\otimes a, y\otimes b)=(x,y)bd(a) \quad \forall\; x,y\in \gfin,\, a,b\in A. 
\]
Using the invariance of $(\cdot,\cdot)$ and the fact that in $\mfz(A)$ we have 
\begin{equation*}
d(abe)=0=ab \cdot d(e)+ae\cdot d(b)+be \cdot d(a)\quad \forall \; a,b,e\in A,  
\end{equation*}
one readily concludes that $\veps$ satisfies the cocycle equation
\begin{equation*}
\veps(x\otimes a,[y\otimes b,z\otimes e])+\veps(y\otimes b,[z\otimes e,x\otimes a])+\veps(z\otimes e,[x\otimes a,y\otimes b])=0,
\end{equation*} 
for all $x,y,z\in \gfin$ and $a,b,e\in A$. It follows that the vector space 
\begin{equation*}
\mathfrak{u}(A):= (\gfin\otimes A) \oplus \mfz(A)
\end{equation*}
admits the structure of a Lie algebra with bracket given by $[\mathfrak{u}(A),\mfz(A)]=0$ and
\begin{equation}\label{LB-uA}
[x\otimes a,y\otimes b]=[x,y]\otimes ab + (x,y)b d(a) \quad \forall \; x,y\in \gfin,\, a,b\in A.  
\end{equation}
It is clear that $\mathfrak{u}(A)$ is a central extension of $\gfin\otimes A$. In fact, we have the following remarkable result due to Kassel \cite{Kas84}*{Thm.~3.3} (see also \cite{MRY90}*{Prop.~2.2}). 
\begin{proposition}\label{P:Kassel} 
$\mathfrak{u}(A)$ is isomorphic to the universal central extension of $\gfin\otimes A$: 
\begin{equation*}
\mathfrak{u}(A)\cong \uce(\gfin\otimes A).
\end{equation*}
\end{proposition}
%
%

\subsection{Gradings on \texorpdfstring{$\mathfrak{u}(A)$}{u(A)}}\label{ssec:aff-I.5}
If in addition $A=\bigoplus_{k\in \Z}A_k$ is a $\Z$-graded algebra, 
then the Lie algebra $\gfin\otimes A$ is naturally $\Z$-graded, with $k$-th graded component
\begin{equation*}
(\gfin\otimes A)_k=\gfin\otimes A_k \quad \forall\; k\in \Z. 
\end{equation*}

The grading on $A$ also naturally induces a $\Z$-graded $A$-module structure on $\Omega(A)$, compatible with that on $A\otimes A$. As the subspace $d(A)$ is itself graded, $\mfz(A)$ is a graded 
$\C$-vector space. By \eqref{LB-uA}, it follows that $\mathfrak{u}(A)$ inherits the structure of a $\Z$-graded Lie algebra, with $k$-th graded component 
\begin{equation*}
\mathfrak{u}(A)_k=(\gfin\otimes A)_k\oplus \mfz(A)_k, 
\end{equation*}
where $\mfz(A)_k$ is the $k$-th graded component of $\mfz(A)$. 
If $A_k=\{0\}$ for all $k<0$ (that is, if $A$ is $\N$-graded), then we may complete $\mathfrak{u}(A)$ with respect to its induced $\N$-grading to obtain a Lie algebra 
\begin{equation*}
\widehat{\mathfrak{u}(A)}=(\gfin\otimes \prod_{k\in \N}A_k) \oplus \prod_{k\in \N}\mfz(A)_k,
\end{equation*}
with  Lie bracket determined by \eqref{LB-uA}, for $a,b\in \wh A=\prod_{k\in \N}A_k$, together with the requirement that
$\prod_{k\in \N}\mfz(A)_k$ be central. 

The next result we will need concerns the functorial nature of $\mathfrak{u}(A)$ and its compatibility with the above completion process. In what follows, $A$ and $B$ are associative, commutative $\C$-algebras, with $A$ taken to be $\N$-graded, as above. Additionally, let us assume we are given an algebra 
homomorphism 
\begin{equation*}
\gamma: B\to \wh{A}.
\end{equation*}
\begin{lemma}\label{L:uA-fun} 
There is a unique Lie algebra homomorphism
\begin{equation*}
\phi_\gamma:\mathfrak{u}(B)\to \widehat{\mathfrak{u}(A)}
\end{equation*}
with the property that $\phi_\gamma|_{\gfin \otimes B}=\id\otimes \gamma$. Explicitly, $\phi_\gamma|_{\mfz(B)}$ is given by 
\begin{equation}\label{g-phi:gen}
\phi_\gamma(bd(e))=\gamma(b)d(\gamma(e)) \quad \forall \; b,e\in B.
\end{equation}
\end{lemma}
\begin{proof}
The bracket relation \eqref{LB-uA} implies that, if $\phi_\gamma$ exists, then it must satisfy \eqref{g-phi:gen}. Conversely, if \eqref{g-phi:gen} determines a well-defined map 
\begin{equation*}
\mfz(B)\to \wh{\mfz(A)}:=\prod_{k\in \N}\mfz(A)_k,
\end{equation*}
 then \eqref{LB-uA} implies that $\phi_\gamma$, uniquely determined by $\phi_\gamma|_{\gfin \otimes B}=\id\otimes \gamma$ and \eqref{g-phi:gen}, will be a Lie algebra homomorphism. 

Since the natural quotient map $A\otimes{A}\to \mfz(A)$ is $\N$-graded, it induces a linear map 
$A\widehat\otimes A\to  \wh{\mfz(A)}$, where $A\widehat\otimes A$ is the completion of $A\otimes A$ with respect to its $\N$-grading. The composition 
\begin{equation*}
B^{\otimes 2} \xrightarrow{\gamma\otimes \gamma} \wh{A}^{\otimes 2}\into A\widehat\otimes A \to \wh{\mfz(A)}
\end{equation*}
then sends $b\otimes e$ to $\gamma(b)d(\gamma(e))$ and factors through $\Omega(B)$ and $\mfz(B)=\Omega(B)/d(B)$, as desired. \qedhere
\end{proof}
\subsection{The Moody--Rao--Yokonuma isomorphism}\label{ssec:aff-II}

We now narrow our focus to the special case where $A$ is the $\Z$-graded algebra $\C[v^{\pm 1},t^{\pm 1}]$ or the $\N$-graded algebra $\C[v^{\pm 1},t]$, where $\deg t=1$ and $\deg v=0$. 

 As in \cite{MRY90}*{\S2}, one finds that $A$-module $\Omega(A)$ is freely generated by $d(v)$ and $d(t)$. Using that, in $\mfz(A)$, we have
\[ 
0=d(v^rt^s)=sv^rt^{s-1}d(t)+rt^sv^{r-1}d(v),
\]
one deduces (\textit{cf}. \cites{MRY90,Enr03}) that $\mfz(A)$ admits the vector space decomposition
\begin{equation*}
\mfz({\C[v^{\pm 1},t^{\pm 1}]})=\!\bigoplus_{(r,s)\in \Z\times \Z}\hspace{-.75em}\C \mathrm{K}_{r,s}\oplus \C c_v \oplus \C c_t, 
\quad
\mfz(\C[v^{\pm 1},t])=\!\bigoplus_{(r,s)\in \Z\times \N_+}\hspace{-.75em}\C \mathrm{K}_{r,s}\oplus \C c_v,
\end{equation*}
where $\mathrm{K}_{0,0}:=0$, and for $(r,s)\in \Z\times \Z^\times$ we have
\begin{equation}\label{z-basis}
\mathrm{K}_{r,s}=\frac{1}{s}v^{r-1}t^sd(v), \quad \mathrm{K}_{s,0}=-\frac{1}{s}v^s t^{-1}d(t),\quad c_v=v^{-1}d(v),\quad c_t=t^{-1}d(t). 
\end{equation} 
By \eqref{GRW-4.7} and Proposition \ref{P:Kassel}, we can (and will) identify $\uce(\dot\mfg[t^{\pm 1}])$ and $\uce(\dot\mfg[t])$ with the Lie algebras $\mathfrak{u}(\C[v^{\pm 1},t^{\pm 1}])$ and $\mathfrak{u}(\C[v^{\pm 1},t])$, respectively. In particular, 
\begin{equation*}
\uce(\dot\mfg[t^{\pm 1}])= \gfin[v^{\pm 1},t^{\pm 1}]\oplus \mfz(\C[v^{\pm 1},t^{\pm 1}])
\end{equation*}
as a vector space, with Lie structure such that $\mfz(\C[v^{\pm 1},t^{\pm 1}])$ is central and
\begin{equation*}
[x\otimes v^r t^s, y\otimes v^kt^\ell]=[x,y]\otimes v^{r+k}t^{s+\ell}+(x,y)v^kt^\ell d(v^r t^s)
\end{equation*}
for all $x,y\in \gfin$ and $r,s,k,\ell\in \Z$. Moreover, in terms of the basis \eqref{z-basis}, we have 
\begin{equation*}
v^kt^\ell d(v^r t^s)=\delta_{r,-k}\delta_{s,-\ell}(rc_v+sc_t)+(r\ell-sk)\mathrm{K}_{r+k,s+\ell}, 
\end{equation*}
By definition, this is equivalent to
\begin{equation*}
v^kt^\ell d(v^r t^s)=\begin{cases} 
  \left(\frac{r\ell-sk}{s+\ell}\right)v^{r+k-1}t^{s+\ell}d(v) \; &\text{ if }\; s\neq -\ell, \\ 
  \delta_{r,-k}r v^{-1} d(v) + s v^{r+k} t^{-1}d(t) \; &\text{ if }\; s=-\ell. 
  \end{cases}
\end{equation*}
The Lie algebra $\uce(\dot\mfg[t])$ may then be characterized as the Lie subalgebra 
\begin{equation*}
\uce(\dot\mfg[t])=\gfin[v^{\pm 1},t]\oplus \hspace{-1em} \bigoplus_{(r,s)\in \Z\times \N_+}\hspace{-.75em}\C \mathrm{K}_{r,s}\oplus \C c_v \subset \uce(\dot\mfg[t^{\pm 1}]).
\end{equation*}
%
%
As a consequence of the general discussion in  \S\ref{ssec:aff-I.5}, $\uce(\dot\mfg[t^{\pm 1}])$  and $\uce(\dot\mfg[t])$ are $\Z$ and $\N$-graded Lie algebras, respectively, with gradings determined
by $\deg t=1$.  

In order to make precise the isomorphisms of \eqref{st-uce}, let us specify $\mfg$ to be rank $\ell+1$, with $\mbI$ taken to be $\{0,\ldots,\ell\}$ so that $\bar\mbI=\{1,\ldots,\ell\}$ 
labels the simple roots of $\gfin$. Let $x_\theta^\pm\in \gfin_{\pm\theta}$ be such that $(x_\theta^+,x_\theta^-)=1$, where $\theta$ is the highest root of $\gfin$. 

The following result is a translation of \cite{MRY90}*{Prop.~3.5}. It appears in the form below in \cite{GRWvrep}; see Propositions 4.4 and 4.7 therein. Recall that $\mft_\kappa$ is the one-dimensional central extension of $\mft$ introduced in Definition \ref{D:t-kappa}. 
\begin{proposition}\label{P:MRY}
The assignment 
\begin{gather*}
\mathrm{K}\mapsto c_t,\quad 
X_{ir}^\pm\mapsto x_i^\pm \otimes t^r, 
\quad 
X_{0r}^\pm \mapsto x^\mp_{\theta}\otimes v^{\pm 1} t^r 
\quad \forall\; i\in \bar\mbI,\; r\in \Z
\end{gather*}
uniquely extends to an isomorphism of $\Z$-graded Lie algebras 
\begin{equation*}
\psi:\mft_\kappa\iso \uce(\dot\mfg[t^{\pm 1}]). 
\end{equation*} 
Moreover, $\psi$ induces isomorphisms of graded Lie algebras 
\begin{equation*}
\psi|_\mfs: \mfs\iso \uce(\dot\mfg[t])
\quad \text{ and }\quad \psi_\mft: \mft\iso \uce(\dot\mfg[t^{\pm 1}])/\C c_t.
\end{equation*}

\end{proposition}
%
%

%

\subsection{Injectivity of \texorpdfstring{$\phi$}{phi}}\label{ssec:aff-III} We now combine the results collected in  Sections \ref{ssec:aff-I}--\ref{ssec:aff-II} to prove that $\phi:\mft\to \wh{\mfs}$ is injective when $\mfg$ is of untwisted affine type. 

Since $\psi|_\mfs$ from Proposition \ref{P:MRY} is graded, it extends to an isomorpism 
\begin{equation*}
\wh{\psi|_\mfs}: \wh\mfs\iso \uce\wh{(\dot\mfg[t])}=\wh{\mathfrak{u}(\C[v^{\pm 1},t])}.
\end{equation*}
As illustrated in Section \ref{ssec:aff-I.5}, the right-hand side above may be identified with 
\begin{gather*}
(\gfin\otimes \C[v^{\pm 1}][\![t]\!])\oplus \prod_{s\in \N}\mfz(\C[v^{\pm 1},t])_s,\\
\text{where} \quad \mfz(\C[v^{\pm 1},t])_s=
\begin{cases}
\C c_v \; &\text{ if }\; s=0,\\
 \bigoplus_{r\in \Z}\C\mathrm{K}_{r,s} \; &\text{ if }\; s\in \N_+.
\end{cases} 
\end{gather*}

Next, let us us extend $\gamma:\C[w^{\pm 1}]\into \C[\![t]\!]$ of \eqref{ext-limit} to a homomorphism  
\begin{equation*}
\gamma: \C[v^{\pm 1},w^{\pm 1}]\into \widehat{\C[v^{\pm 1},t]}=\C[v^{\pm 1}][\![t]\!] 
\end{equation*}
by setting $\gamma(v)=v$. By Lemma \ref{L:uA-fun}, $\id \otimes \gamma:\gfin[v^{\pm 1},w^{\pm 1}]\to \gfin\otimes \C[v^{\pm 1}][\![t]\!]$ extends uniquely to a homomorphism of Lie algebras 
\begin{gather*}
\phi_\gamma:\uce(\dot\mfg[w^{\pm 1}])\to \uce\widehat{(\dot\mfg[t])}\\
\text{with }\; \phi_\gamma(fdg)=\gamma(f)d(\gamma(g)) \quad \forall \; f,g\in \C[v^{\pm 1},w^{\pm 1}].
\end{gather*}
The following proposition completes our proof that $\phi:\mft\to \wh{\mfs}$ is injective when $\mfg$ is of untwisted affine type, and therefore completes the proof of Theorem \ref{T:PBW}. 
\begin{proposition}\label{P:CL-expl} 
Let $\wh{\psi|_\mfs}$ and $\phi_\gamma$ be as above. Then:
\begin{enumerate}[font=\upshape]

\item\label{CL-expl:2}  $\Ker(\phi_\gamma)=\C c_w$. Consequently, $\phi_\gamma$ induces an injection
\begin{equation*}
\bar\phi_\gamma: \uce(\dot\mfg[w^{\pm 1}])/\C c_w \into \uce\wh{(\dot\mfg[t])}. 
\end{equation*} 

\item\label{CL-expl:3}  $\phi:\mft\to \wh{\mfs}$ is injective, and satisfies $\phi=(\wh{\psi|_\mfs})^{-1}\circ \bar\phi_\gamma\circ \psi_\mft$.
\end{enumerate}
\end{proposition}
\begin{proof}

Let us begin by establishing that the kernel of $\phi_\gamma$ coincides with $\C c_w$. 

\begin{proof}[Proof of \eqref{CL-expl:2}]\let\qed\relax

Since $\phi_\gamma|_{\gfin[v^{\pm 1},w^{\pm 1}]}=\id \otimes \gamma$ is injective and 
\begin{equation*}
\phi_\gamma(\mfz(\C[v^{\pm 1},w^{\pm 1}]))\subset \wh{\mfz(\C[v^{\pm 1},t])},
\end{equation*}
 it suffices to show that the restriction of $\phi_\gamma$ to $\mfz(\C[v^{\pm 1},w^{\pm 1}])$ has kernel $\C c_w$. 

 In $\wh{\mfz(\C[v^{\pm 1},t])}$, we have the relations 
\begin{equation*}
\gamma(v^r w^{-1})d(\gamma(w))=
\sum_{k\geq 0} (-1)^k v^r t^k d(t),\quad 
v^rt^k d(t)= -\frac{r}{k+1}v^{r-1}t^{k+1}d(v). 
\end{equation*}
Hence, 
$\phi_\gamma$ is determined on 
$\mfz(\C[v^{\pm 1},w^{\pm 1}])$ by 
\begin{gather*}
\phi_\gamma(v^rw^s d(v))= v^r(t+1)^s d(v), \\
\phi_\gamma(v^rw^{-1}d(w))=-rv^{r-1}\sum_{k\geq 0} (-1)^k\frac{t^{k+1}}{k+1}d(v)
=
-rv^{r-1}\log(t+1)d(v),
\end{gather*}
for all $r,s\in \Z$. In particular, $\phi_\gamma(c_w)=\phi_\gamma(w^{-1}d(w))=0$. 

Since $\mfz(\C[v^{\pm 1},w^{\pm 1}])$ has basis 
\begin{equation*}
\{v^r w^s d(v)\}_{(r,s)\in\Z\times \Z^\times} \cup \{v^r w^{-1} d(w)\}_{r\in \Z}\cup\{v^{-1}d(v)\},
\end{equation*}
to show that $\Ker(\phi_\gamma)=\C c_w$, it suffices to prove that the set 
\begin{equation*} 
\{v^r(t+1)^sd(v)\}_{(r,s)\in \Z\times \Z^\times}\cup \{v^{r-1}\log(t+1)d(v)\}_{r\in \Z^\times}\cup \{c_v\}
\end{equation*}
is linearly independent in 
\begin{equation*}
\wh{\mfz(\C[v^{\pm 1},t])}=\prod_{s\in \N_+}\mfz(\C[v^{\pm 1},t])_s \oplus \C c_v\cong t\C[v^{\pm 1}][\![t]\!]\oplus \C v^{-1} \subset \C[v^{\pm 1}][\![t]\!],
\end{equation*}
where the embedding of vector spaces $\wh{\mfz(\C[v^{\pm 1},t])}\subset \C[v^{\pm 1}][\![t]\!]$ alluded to above is determined by identifying $v^r t^s d(v)$ with $v^r t^s$.

This follows readily from the observation that $\{(t+1)^s\}_{s\in \Z^\times}\cup\{\log(t+1)\}$ is a linearly independent set in $\C[\![t]\!]$, which can be deduced using the injectivity of $\gamma$ and the observation that
\begin{equation*}
\gamma(w^{s-1})=\begin{cases}
                 \frac{1}{s}\partial_t (t+1)^s \; &\text{ if } s\neq 0,\\
                 \partial_t \log(t+1)\; &\text{ if } s=0.
                \end{cases}
\end{equation*}
\end{proof}

\begin{proof}[Proof of \eqref{CL-expl:3}] It suffices to verify the identity 
$\phi=(\wh{\psi|_\mfs})^{-1}\circ \bar\phi_\gamma\circ \psi_\mft$ on the generating set $\{X_{ir}^\pm\}_{i\in \mbI, r\in \Z}$ of $\mft$. This is easily done directly, using the explicit formulas for $\psi_\mft$ and $\psi_\mfs$ given in Proposition \ref{P:MRY} (see also \eqref{ext-limit}). \qedhere
\end{proof}
\let\qed\relax
\end{proof}

\appendix

\section{Grading completions} \label{App:A}

In this appendix we prove Proposition \ref{P:A-comp} which serves to clarify a number of properties satisfied by the grading completion of any $\N$-graded $\C[\hbar]$-algebra. This proposition has been applied to prove Lemma \ref{L:whYhg} and, though it is elementary, has been included for sake of completeness. 
\begin{proposition}\label{P:A-comp}
Let $\mathrm{A}=\bigoplus_{n\in \N}\mathrm{A}_n$ be a $\N$-graded $\C[\hbar]$-algebra satisfying:
\begin{enumerate}[label=(\alph*), font=\upshape]
\item\label{A-comp:1} $\hbar \mathrm{A}\subset \mathrm{A}_+=\bigoplus_{n>0}\mathrm{A}_n$, 
\item\label{A-comp:2} $\mathrm{A}_+^n=\bigoplus_{k\geq n}\mathrm{A}_k$ for each $n\in \N$.  
\end{enumerate}
Then the formal completion of $\mathrm{A}$ with respect to its $\N$-grading,
\begin{equation*}
\widehat{\mathrm{A}}=\prod_{n\in \N}\!\mathrm{A}_n,
\end{equation*}
is a unital, associative $\C[\![\hbar]\!]$-algebra. Moreover:
\begin{enumerate}[font=\upshape]
\item\label{A:1} The canonical $\C[\hbar]$-algebra homomorphism $\Upsilon:\mathrm{A}\to \varprojlim_n\!\left(\mathrm{A}/\mathrm{A}_+^n\right)$ extends to an isomorphism of $\C[\![\hbar]\!]$-algebras 
\begin{equation*}
\widehat\Upsilon:\widehat{\mathrm{A}}=\prod_{n\in \N}\!\mathrm{A}_n\iso \varprojlim_n\left(\mathrm{A}/\mathrm{A}_+^n\right).
\end{equation*}
\item\label{A:2} $\widehat{\mathrm{A}}$ is separated and complete as a $\C[\![\hbar]\!]$-module.
\item\label{A:3} $\widehat{\mathrm{A}}$ is a torsion free $\C[\![\hbar]\!]$-module, provided $\mathrm{A}$ is a torsion free $\C[\hbar]$-module. 
\end{enumerate}
\end{proposition}
\begin{proof}
First note that the condition \ref{A-comp:1} guarantees that both completions of $\mathrm{A}$ appearing above are unital, associative $\C[\![\hbar]\!]$-algebras. 

To prove \eqref{A:1}, note that \ref{A-comp:2} implies that, for each $n\in \N$, we have 
\begin{equation*}
\widehat{\mathrm{A}}/\widehat{\mathrm{A}}_{\geq n} \cong \mathrm{A}/\mathrm{A}_+^n, \quad \text{ where }\quad \widehat{\mathrm{A}}_{\geq n}=\prod_{k\geq n}\mathrm{A}_k.
\end{equation*}
We thus have a canonical homomorphism of $\C[\![\hbar]\!]$-algebras 
\begin{equation*}
\widehat{\Upsilon}:\widehat{\mathrm{A}}\to \varprojlim_{n}\left(\widehat{\mathrm{A}}/\widehat{\mathrm{A}}_{\geq n}\right)\cong \varprojlim_n\left(\mathrm{A}/\mathrm{A}_+^n\right). 
\end{equation*}
Moreover, the composite of $\widehat\Upsilon$ with the inclusion $\mathrm{A}\into \widehat{\mathrm{A}}$ coincides with $\Upsilon$.

Under the identifying $\mathrm{A}/\mathrm{A}_+^n\cong \bigoplus_{k=0}^{n-1}\mathrm{A}_k$, the projection $\mathrm{p}_{n+1}:\mathrm{A}/\mathrm{A}_+^{n+1}\to \mathrm{A}/\mathrm{A}_+^{n}$ coincides with the truncation operator  $x_0+\ldots + x_{n}\mapsto x_0+\ldots + x_{n-1}$. We may therefore identify
\begin{equation*}
\varprojlim_n\left(\mathrm{A}/\mathrm{A}_+^n\right)=\{(x_0+\ldots+x_{n-1})_{n\in \N}\,:\, x_k\in \mathrm{A}_k\}\subset \prod_{n\in \N}\left(\bigoplus_{k=0}^{n-1}\mathrm{A}_k\right)
\end{equation*}
Under this identification, we have 
\begin{equation*}
\widehat\Upsilon:\sum_{k\geq 0}x_k\mapsto (x_0+\ldots+x_{n-1})_{n\in \N},
\end{equation*}
from which the bijectivity of $\widehat\Upsilon$ follows immediately.

Consider Part \eqref{A:3}. If $x=\sum_k x_k\in \widehat{\mathrm{A}}$ is such that $\hbar x=0$, then we must have $\hbar x_k=0$ for each $k$. As $\mathrm{A}$ is assumed to be torsion free, we can conclude that each $x_k$, and thus $x$ itself, vanishes. 

It remains to prove \eqref{A:2}. By \ref{A-comp:1}, $\hbar^n\widehat{\mathrm{A}}\subset \widehat{\mathrm{A}}_{\geq n}$, and thus 
\begin{equation*}
\bigcap_{n\in\N}\hbar^n\widehat{\mathrm{A}}\subset \bigcap_{n\in \N}\widehat{\mathrm{A}}_{\geq n}=\{0\}. 
\end{equation*}
Therefore, $\widehat{\mathrm{A}}$ is a separated $\C[\![\hbar]\!]$-module. To show it is also complete, we must argue that the natural homomorphism 
\begin{equation*}
\Theta:\widehat{\mathrm{A}}\to \varprojlim_n\left(\widehat{\mathrm{A}}/\hbar^n\widehat{\mathrm{A}}\right)
\end{equation*}
is surjective. To this end, note that any $x\in\varprojlim_n\left(\widehat{\mathrm{A}}/\hbar^n\widehat{\mathrm{A}}\right)$ may be represented as 
$
x=\left(\sum_{n\geq 0}\mathrm{q}_k(x_{k,n})\right)_{k\in \N}, 
$
where 
\begin{enumerate}[label=(\roman*)]
\item\label{Ah:1} $x_{k,n}\in \mathrm{A}_n$ satisfy 
$
x_{k,n}-x_{\ell,n}\in \hbar^{\min(k,\ell)} \mathrm{A}$ for all $k,\ell,n\in \N.
$
\item $\mathrm{q}_k:\mathrm{A}\to \mathrm{A}/\hbar^k \mathrm{A}$ is the natural quotient map.
\end{enumerate}

Set $x_n=x_{n+1,n}\in \mathrm{A}_n$ for all $n\in \N$. We claim that $x$ is equal to the image of $\sum_{n}x_n$ under $\Theta$. To prove this, it suffices to show that 
\begin{equation*}
x_n-x_{k,n}\in \hbar^k\mathrm{A} \quad \forall \; k,n\in \N. 
\end{equation*}
By \ref{Ah:1}, $x_n-x_{k,n}\in \hbar^{\min(n+1,k)}\mathrm{A}$ for all $k$ and $n$, hence the assertion is true for $k<n$. If $k\geq n$ then, by \ref{A-comp:1} and \ref{A-comp:2}, we have 
\begin{equation*}
x_n-x_{k,n}\in \hbar^{n+1}\mathrm{A}\subset \bigoplus_{k>n} \mathrm{A}_k,
\end{equation*}
which implies that $x_n-x_{k,n}=0\in \hbar^k\mathrm{A}$, as desired. 
\end{proof}
%
%



\begin{bibdiv}
\begin{biblist}

\bib{BerTsy19}{article}{
      author={Bershtein, M.},
      author={Tsymbaliuk, A.},
       title={Homomorphisms between different quantum toroidal and affine
  {Y}angian algebras},
        date={2019},
        ISSN={0022-4049},
     journal={J. Pure Appl. Algebra},
      volume={223},
      number={2},
       pages={867\ndash 899},
}

\bib{Dr}{article}{
      author={Drinfel'd, V.},
       title={Hopf algebras and the quantum {Y}ang-{B}axter equation},
        date={1985},
     journal={Soviet Math. Dokl.},
      volume={32},
      number={1},
       pages={254\ndash 258},
}

\bib{DrQG}{inproceedings}{
      author={Drinfel'd, V.},
       title={Quantum groups},
        date={1987},
   booktitle={Proceedings of the {I}nternational {C}ongress of
  {M}athematicians, {V}ol. 1, 2 ({B}erkeley, {C}alif., 1986)},
   publisher={Amer. Math. Soc., Providence, RI},
       pages={798\ndash 820},
}

\bib{Enr03}{article}{
      author={Enriquez, B.},
       title={P{BW} and duality theorems for quantum groups and quantum current
  algebras},
        date={2003},
        ISSN={0949-5932},
     journal={J. Lie Theory},
      volume={13},
      number={1},
       pages={21\ndash 64},
}

\bib{EnQHA}{article}{
      author={Enriquez, B.},
       title={Quasi-{H}opf algebras associated with semisimple {L}ie algebras
  and complex curves},
        date={2003},
        ISSN={1022-1824},
     journal={Selecta Math. (N.S.)},
      volume={9},
      number={1},
       pages={1\ndash 61},
}

\bib{FiTs19}{article}{
      author={Finkelberg, M.},
      author={Tsymbaliuk, A.},
       title={Shifted quantum affine algebras: integral forms in type {$A$}},
        date={2019},
        ISSN={2199-6792},
     journal={Arnold Math. J.},
      volume={5},
      number={2-3},
       pages={197\ndash 283},
}

\bib{GTL1}{article}{
      author={Gautam, S.},
      author={Toledano~Laredo, V.},
       title={Yangians and quantum loop algebras},
        date={2013},
        ISSN={1022-1824},
     journal={Selecta Math. (N.S.)},
      volume={19},
      number={2},
       pages={271\ndash 336},
}

\bib{GTL2}{article}{
      author={Gautam, S.},
      author={Toledano~Laredo, V.},
       title={Yangians, quantum loop algebras, and abelian difference
  equations},
        date={2016},
        ISSN={0894-0347},
     journal={J. Amer. Math. Soc.},
      volume={29},
      number={3},
       pages={775\ndash 824},
}

\bib{GTLW19}{article}{
      author={Gautam, S.},
      author={Toledano~Laredo, V.},
      author={Wendlandt, C.},
       title={The meromorphic {R}-matrix of the {Y}angian},
     journal={Progr. Math.},
        date={2021},
      status={to appear},
        note={\tt arXiv:1907.03525},
}

\bib{GWPoles}{unpublished}{
      author={Gautam, S.},
      author={Wendlandt, C.},
       title={Poles of finite-dimensional representations of {Y}angians},
        date={2020},
        note={\tt arXiv:2009.06427},
}

\bib{Gu05}{article}{
      author={Guay, N.},
       title={Cherednik algebras and {Y}angians},
        date={2005},
        ISSN={1073-7928},
     journal={Int. Math. Res. Not. IMRN},
      number={57},
       pages={3551\ndash 3593},
}

\bib{Gu07}{article}{
      author={Guay, N.},
       title={Affine {Y}angians and deformed double current algebras in type
  {A}},
        date={2007},
        ISSN={0001-8708},
     journal={Adv. Math.},
      volume={211},
      number={2},
       pages={436\ndash 484},
}

\bib{Gu09}{article}{
      author={Guay, N.},
       title={Quantum algebras and quivers},
        date={2009},
        ISSN={1022-1824},
     journal={Selecta Math. (N.S.)},
      volume={14},
      number={3-4},
       pages={667\ndash 700},
}

\bib{GM}{article}{
      author={Guay, N.},
      author={Ma, X.},
       title={From quantum loop algebras to {Y}angians},
        date={2012},
        ISSN={0024-6107},
     journal={J. Lond. Math. Soc. (2)},
      volume={86},
      number={3},
       pages={683\ndash 700},
}

\bib{GNW}{article}{
      author={Guay, N.},
      author={Nakajima, H.},
      author={Wendlandt, C.},
       title={Coproduct for {Y}angians of affine {K}ac-{M}oody algebras},
        date={2018},
        ISSN={0001-8708},
     journal={Adv. Math.},
      volume={338},
       pages={865\ndash 911},
}

\bib{GRWEquiv}{article}{
      author={Guay, N.},
      author={Regelskis, V.},
      author={Wendlandt, C.},
       title={Equivalences between three presentations of orthogonal and
  symplectic {Y}angians},
        date={2019},
        ISSN={0377-9017},
     journal={Lett. Math. Phys.},
      volume={109},
      number={2},
       pages={327\ndash 379},
}

\bib{GRWvrep}{article}{
      author={Guay, N.},
      author={Regelskis, V.},
      author={Wendlandt, C.},
       title={Vertex representations for {Y}angians of {K}ac-{M}oody algebras},
        date={2019},
     journal={J. \'{E}c. polytech. Math.},
      volume={6},
       pages={665\ndash 706},
}

\bib{Io96}{article}{
      author={Iohara, K.},
       title={Bosonic representations of {Y}angian double
  {$\mathcal{D}Y_{\hslash}(\mathfrak{g})$} with
  {$\mathfrak{g}=\mathfrak{gl}_N,\mathfrak{sl}_N$}},
        date={1996},
        ISSN={0305-4470},
     journal={J. Phys. A},
      volume={29},
      number={15},
       pages={4593\ndash 4621},
}

\bib{JKMY18}{article}{
      author={Jing, N.},
      author={Ko\v{z}i\'{c}, S.},
      author={Molev, A.},
      author={Yang, F.},
       title={Center of the quantum affine vertex algebra in type {$A$}},
        date={2018},
        ISSN={0021-8693},
     journal={J. Algebra},
      volume={496},
       pages={138\ndash 186},
}

\bib{JiYaLi20}{article}{
      author={Jing, N.},
      author={Yang, F.},
      author={Liu, M.},
       title={Yangian doubles of classical types and their vertex
  representations},
        date={2020},
        ISSN={0022-2488},
     journal={J. Math. Phys.},
      volume={61},
      number={5},
       pages={051704, 39},
}

\bib{KacBook90}{book}{
      author={Kac, V.},
       title={Infinite-dimensional {L}ie algebras},
     edition={Third edition},
   publisher={Cambridge University Press, Cambridge},
        date={1990},
        ISBN={0-521-37215-1; 0-521-46693-8},
}

\bib{Kas84}{inproceedings}{
      author={Kassel, C.},
       title={K\"{a}hler differentials and coverings of complex simple {L}ie
  algebras extended over a commutative algebra},
        date={1984},
   booktitle={Proceedings of the {L}uminy conference on algebraic {$K$}-theory
  ({L}uminy, 1983)},
      volume={34},
       pages={265\ndash 275},
}

\bib{KasBook95}{book}{
      author={Kassel, C.},
       title={Quantum groups},
      series={Graduate Texts in Mathematics},
   publisher={Springer-Verlag, New York},
        date={1995},
      volume={155},
        ISBN={0-387-94370-6},
}


\bib{Khor95}{article}{
author={Khoroshkin, S.},
title={Central extension of the {Y}angian double},
conference={
title={Alg\`ebre non commutative, groupes quantiques et invariants},
address={Reims},
date={1995},
},
book={
series={S\'{e}min. Congr.},
volume={2},
publisher={Soc. Math. France, Paris},
date={1997},
},
pages={119--135},
}

\bib{KT96}{article}{
      author={Khoroshkin, S.},
      author={Tolstoy, V.},
       title={Yangian double},
        date={1996},
        ISSN={0377-9017},
     journal={Lett. Math. Phys.},
      volume={36},
      number={4},
       pages={373\ndash 402},
}

\bib{Kod18}{article}{
      author={Kodera, R.},
       title={Higher level {F}ock spaces and affine {Y}angian},
        date={2018},
        ISSN={1083-4362},
     journal={Transform. Groups},
      volume={23},
      number={4},
       pages={939\ndash 962},
}

\bib{Kod19}{article}{
      author={Kodera, R.},
       title={Affine {Y}angian action on the {F}ock space},
        date={2019},
        ISSN={0034-5318},
     journal={Publ. Res. Inst. Math. Sci.},
      volume={55},
      number={1},
       pages={189\ndash 234},
}

\bib{Kod19b}{article}{
      author={Kodera, R.},
       title={Braid group action on affine {Y}angian},
        date={2019},
     journal={SIGMA Symmetry Integrability Geom. Methods Appl.},
      volume={15},
       pages={Paper No. 020, 28},
}

\bib{LevPBW}{article}{
      author={Levendorski\u{\i}, S.},
       title={On {PBW} bases for {Y}angians},
        date={1993},
        ISSN={0377-9017},
     journal={Lett. Math. Phys.},
      volume={27},
      number={1},
       pages={37\ndash 42},
}

\bib{Mont}{book}{
      author={Montgomery, S.},
       title={Hopf algebras and their actions on rings},
      series={CBMS Regional Conference Series in Mathematics},
   publisher={Published for the Conference Board of the Mathematical Sciences,
  Washington, DC; by the American Mathematical Society, Providence, RI},
        date={1993},
      volume={82},
        ISBN={0-8218-0738-2},
}

\bib{MRY90}{article}{
      author={Moody, R.~V.},
      author={Rao, S.~E.},
      author={Yokonuma, T.},
       title={Toroidal {L}ie algebras and vertex representations},
        date={1990},
        ISSN={0046-5755},
     journal={Geom. Dedicata},
      volume={35},
      number={1-3},
       pages={283\ndash 307},
}

\bib{Naz20}{article}{
      author={Nazarov, M.},
       title={Double {Y}angian and the universal {$R$}-matrix},
        date={2020},
        ISSN={0289-2316},
     journal={Jpn. J. Math.},
      volume={15},
      number={1},
       pages={169\ndash 221},
}

\bib{Tsy17b}{article}{
      author={Tsymbaliuk, A.},
       title={The affine {Y}angian of {$\mathfrak{gl}_1$} revisited},
        date={2017},
        ISSN={0001-8708},
     journal={Adv. Math.},
      volume={304},
       pages={583\ndash 645},
}

\bib{Tsy17}{article}{
      author={Tsymbaliuk, A.},
       title={Classical limits of quantum toroidal and affine {Y}angian
  algebras},
        date={2017},
        ISSN={0022-4049},
     journal={J. Pure Appl. Algebra},
      volume={221},
      number={10},
       pages={2633\ndash 2646},
}

\bib{WRQD}{unpublished}{
      author={Wendlandt, C.},
       title={The restricted quantum double of the {Y}angian},
        note={In preparation},
}

\bib{YaGu1}{article}{
      author={Yang, Y.},
      author={Zhao, G.},
       title={The cohomological {H}all algebra of a preprojective algebra},
        date={2018},
        ISSN={0024-6115},
     journal={Proc. Lond. Math. Soc. (3)},
      volume={116},
      number={5},
       pages={1029\ndash 1074},
}

\bib{YaGu3}{article}{
      author={Yang, Y.},
      author={Zhao, G.},
       title={Cohomological {H}all algebras and affine quantum groups},
        date={2018},
        ISSN={1022-1824},
     journal={Selecta Math. (N.S.)},
      volume={24},
      number={2},
       pages={1093\ndash 1119},
}

\bib{YaGuPBW}{article}{
   author={Yang, Y.},
   author={Zhao, G.},
   title={The PBW theorem for affine Yangians},
   journal={Transform. Groups},
   volume={25},
   date={2020},
   number={4},
   pages={1371--1385},
   issn={1083-4362},
}

\end{biblist}
\end{bibdiv}
\end{document}